\documentclass[11pt]{article}

\usepackage{amsmath}
\usepackage{amsmath}
\usepackage{amsthm}
\usepackage{amsfonts}
\usepackage{amssymb}
\usepackage{color}
\usepackage{graphicx}
\usepackage[colorlinks=true,linkcolor=blue,urlcolor=blue,citecolor=red, hyperfigures=false]{hyperref}
\usepackage{dsfont}
\usepackage{enumerate}

\let\inf\relax \DeclareMathOperator*\inf{\vphantom{p}inf}
\let\max\relax \DeclareMathOperator*\max{\vphantom{p}max}
\let\min\relax \DeclareMathOperator*\min{\vphantom{p}min}

\newcommand{\N}{\mathbb N}

\newcommand{\R}{\mathbb R}

\newcommand{\D}{\mathbb D}
\newcommand{\E}{\mathbb E}

\def\P{\mathbb P}
\def\F{\mathbb F}

\newcommand{\be}{\begin{equation}}
\newcommand{\ee}{\end{equation}}
\def\1{{\bf 1}}

\def\Dt0{{\bf D}(t_0)}

\def\pa{\partial}

\def\tr{\hbox{tr}}

\newcommand{\ind}{{\mathds{1}}}

\definecolor{ProcessBlue}{cmyk}{1,0,0,0.40}
\def\blu{\color{ProcessBlue}}
\definecolor{PineGreen}{cmyk}{0.8,0,1,0.6}

\def\di{\mbox{div}}

\newcommand{\ba}{\[\begin{array}{rl}}
\newcommand{\ea}{\end{array}\]}
\newcommand{\bea}{\begin{eqnarray}}
\newcommand{\eea}{\end{eqnarray}}
\newcommand{\beaa}{\begin{eqnarray*}}
\newcommand{\eeaa}{\end{eqnarray*}}
\numberwithin{equation}{section}
\numberwithin{figure}{section}
\theoremstyle{plain}
\newtheorem{Theorem}{Theorem}[section]
\newtheorem{Definition}[Theorem]{Definition}
\newtheorem{Proposition}[Theorem]{Proposition}
\newtheorem{Lemma}[Theorem]{Lemma}

\newtheorem{Remark}[Theorem]{Remark}
\newtheorem{Remarks}[Theorem]{Remarks}

\newcommand{\BR}{{\cal B}}

\newcommand{\FR}{{\cal F}}

\newcommand{\LR}{{\cal L}}
\newcommand{\MR}{{\cal M}}

\newcommand{\PR}{{\cal P}}

\newcommand{\TR}{{\cal T}}

\newcommand{\noi}{\noindent}

\def\pa{\vskip4truept \noindent}

\begin{document}

\title{A two player zerosum game where only one player observes a Brownian motion.}

\author{Fabien Gensbittel$^{(1)}$ and Catherine Rainer$^{(2)}$\\
$\;$\\
\small {(1)} TSE (GREMAQ, UT1),\\
\small {(2)} Universit\'{e} de Bretagne Occidentale, 6, avenue Victor-le-Gorgeu, B.P. 809, 29285 Brest cedex, France\\
 \small e-mail : 
Fabien.Gensbittel@ut-capitole.fr, Catherine.Rainer@univ-brest.fr}

  \maketitle

\vspace{3mm}


\begin{abstract}
We study a two-player zero-sum game in continuous time, where the payoff -a running cost- depends on a Brownian motion. This Brownian motion is observed in real time by one of the players. The other one observes only the actions of  his/her opponent. We prove that the game has a value and characterize it as the largest convex subsolution of a Hamilton-Jacobi equation on the space of probability measures.
\end{abstract}

\vspace{3mm}

\vspace{3mm}

\noindent{\bf A.M.S. classification :} 91A05, 91A23, 49N70
\vspace{3mm}

\section{Introduction.}

In this paper, we consider a two-player zero-sum game in continuous time. Given some finite time horizon $T>0$, some initial time $t\in[0,T]$ and a probability measure $m$ on $\R^d$, the payoff consists in a running cost of type
\[ \E_m\left[\int_t^Tf(s,B_s,u_s,v_s)ds\right],\]
where $(u_s)$ (resp. $(v_s)$) is the control played by the first (resp. second) player, and $(B_s)$ is a standard $\R^d$-valued Brownian motion, which, under $\P_m$, starts with the initial law $m$ at time $t$. Player 1 wants to minimize this payoff, while player 2 wants to maximize it. This is a game with asymmetric information: the first player observes in real time the Brownian motion and the controls of the second player, while the second player cannot see neither the Brownian motion nor the payoff of the game, but only the controls of  the first player.
The game is approximated by a sequence of discrete time games with vanishing time increments, i.e.  where the players play more and more frequently. \\
A first result is that, under Isaacs' assumption, the sequence of the values for the discrete games converges to some function $(t,m)\mapsto V(t,m)$ which is equal to the value of a control problem 
\begin{equation}
\label{measproc}
 \inf_{M\in\MR(t,m)}\E[\int_t^TH(s,M_s)ds],
 \end{equation}
where $H(t,m)=\inf_u\sup_v\int_{\R^d}f(t,x,u,v)dm(x)$, and $\MR(t,m)$ is the set of measure valued processes $(M_s)_{s\in[t,T]}$ which satisfy
\begin{itemize}
\item $M_t=m$,
\item for all $t\leq t_1\leq t_2\leq T$ and all continuous, bounded function $\phi$, 
\[ \E\left[\int_{\R^d}\phi dM_{t_2}|\FR^M_{t_1}\right]=\int_{\R^d}\phi dp^{t_1,M_{t_1}}_{t_2},\]
with $p^{t,m}_r$ the law at time $r$ of the Brownian motion $(B_s)$ starting with law $m$ at time $t$ and $\FR^M$ the filtration generated by $M$.
\end{itemize}
Our second and main result is that this value function $V$ can be characterized as the largest bounded and continuous function of $(t,m)$ which is convex in $m$ and subsolution of the following equation :
\begin{equation}
\label{hjiintro}
\left\{
\begin{array}{l}
\partial_tU(t,m)+\frac{1}{2}\int_{\R^d}\di[D_mU](t,m,x)m(dx)+H(t,m)=0,\; (t,m)\in[0,T]\times\PR_2,\\
\\
U(T,m)=0,\; m\in\PR_2.
\end{array}
\right.
\end{equation}
Here the derivative with respect to the measure $m$, $D_mU$, is defined in Cardaliaguet-Delarue-Lasry-Lions \cite{CDLL}. Concerning the notion of subsolution, in this case, a naive extension of the classical notion of viscosity subsolution as it can be found in Crandall-Ishii-Lions \cite{CIL} is sufficient. \\

Dynamic games with asymmetric information were introduced in the framework of repeated games by Aumann and Maschler in the 1960th (see \cite{AumannMaschler}) and much later -in 2007- in continuous time by Cardaliaguet.
 In his seminal paper \cite{c1}, the asymmetrically observed parameters belong to some finite sets $I$ and $J$ and are fixed before the game starts. It is shown that the game has a value which is solution in some dual sense of a Hamilton-Jacobi-Isaacs' equation (see  Cardaliaguet-Rainer \cite{cr1} for a generalisation to stochastic differential games and  Oliu-Barton \cite{OliuBarton}
 for  correlated information).\\
In Cardaliaguet-Rainer \cite{cr2}, in the case of only one non-informed player and without dynamics, an alternative formulation of type \eqref{measproc}, in terms of an optimization problem over the belief process of the uninformed player, is given (see Gr\"un \cite{Gruen} for stochastic differential games and  \cite{GR} by the authors for its extension to lack of information on both sides). This alternative formulation permits firstly to derive some optimal strategy for the informed player. Further, in cases where the dual approach of \cite{c1} isn't possible (typically when the distribution of the asymmetrically observed parameters has a continuous support), it provides an new angle of attack for the PDE-characterization. \\ 
However, since the reinterpretation in terms of a control problem with respect to  measure-valued processes is possible only for the upper value of the game, it doesn't permit to prove the existence of a value. 
This is one of the reasons why we focus here on the interpretation of the continuous game as a limit of a sequence of discrete time games. 
On the other hand, it permits us to contribute to the growing interest for the interactions between discrete and continuous time games: Rather simultaneously Cardaliaguet-Laraki-Sorin \cite{CLS} and Neyman \cite{Neyman} firstly investigate this area, showing that continuous time games may be interpreted as limits of discrete time games where the players play more and more faster (see also a very recent  paper of Sorin \cite{Sorin}). This approach is then used in Cardaliaguet-Rainer-Rosenberg-Vieille \cite{CRRV} and Gensbittel \cite{Gensbittel} for the asymmetric observation of a Markov process. 
 These last two references are concerned with the asymmetrical observation in real time  of a continuous time, random process. In both, it is a Markov process with finite state space.  Our paper provides a first step to the observation of a Markov process with continuous support.\\
 But several problems remain open: Since, in the present work, the associated PDE has a classical solution, there is no need of a comparison theorem for viscosity solutions in the measure space,  for which a suitable definition is still to find. In the general case of the observation of an arbitrary diffusion process, one cannot avoid this difficulty. The second open question is how to tackle directly the continuous game and the existence of its value.\\
 
The paper is organized as follows: In Section \ref{notations}, we introduce the notations relative to the Wasserstein space, the Brownian motion and Gaussian kernels and introduce the continuous time game and the approximating discrete time games. In section \ref{alternative}, the alternative formulation \eqref{measproc} is established. Finally, in section \ref{edp}, we introduce the PDE \eqref{hjiintro} and state and prove the characterization of the value of the game.

\section{Model: notations and reminders.}
\label{notations}

\subsection{Wasserstein space and Wasserstein distance.}

\pa 
For fixed $d$, let $\PR$ be the set of all probability measures on $\R^d$. 
On $\PR$,  we introduce, for all $p\geq 1$ the normalized $p$th moment
\[ |m|_p=\left(\int_{\R^d}|x|^pdm(x)\right)^{\frac{1}{p}}.\]
The associated subspaces are :
\[ \PR_p=\{ m \in\PR,|m|_p<\infty\}, p\geq 1\]
For all $p\geq 1$, we define on $\PR_p$ the $p$-Wasserstein distance:
\[ d_p(m,m')=\inf_{\pi\in\Gamma(m,m')}\left(\int_{\R^d\times\R^d}|x-y|^p\pi(dx,dy)\right)^{\frac{1}{p}}, m,m'\in\PR_p\]
where $\Gamma(m,m')$ is the set of probability measures $\pi$ on $\R^d \times \R^d$ with first (resp. second) marginal $m$ (resp $m'$).
Recall that for $p=1$, there is a dual formulation:

\[ d_1(m,m')=\sup\left\{ \int_{\R^d}\varphi d(m-m'), \varphi \; 1\mbox{-Lipschitz}\right\}.\]
\pa
For the above result as well as for the statements of the following Lemma \ref{dro} and their proofs, we refer to Villani  \cite{Villani}.
\pa
In this paper, we will place us mainly on the space $\PR_2$, sometimes endowed with the corresponding $d_2$-metric, but more frequently seen as a subset of $\PR_1$, endowed with $d_1$.
\pa 
Let us state some useful standard results.
\begin{Lemma}
\label{dro} \
\begin{enumerate}
\item
For all $p\geq 1$, it holds that, for some well known constant $C_p$,
\[ d_p^p(m,m')\leq C_p(|m|_p^p+|m'|_p^p).\]
\item
Let $m,m'\in\PR_2$. 
We have
\[ d_1(m,m')\leq \int_{\R^d}|x||m-m'|(dx).\]
In particular,
if $m$ (resp. $m'$) has a density $\rho$ (resp. $\rho'$), then
\[ d_1(m,m')\leq \int_{\R^d}|x||\rho-\rho'|(x)dx.\]
\item The balls $\BR_2(R)=\{ m\in\PR,|m|_2\leq R\}$ are $d_1$ compact.
\item For all $\phi$ continuous with $|\phi(x)|\leq \alpha (1+|x|^p)$ for some constant $\alpha$, the map $\PR_p \ni m \mapsto \int_{\R^d} \phi(x) dm(x)$ is $d_p$ continuous. If $\phi$ is non-negative, it is $d_q$ lower semicontinuous (l.s.c.) for all $1\leq q \leq p$.
\item A sequence of measures $(\mu_n)$ converges in $\PR_p$ if and only if it converges weakly and admits uniformly integrable $p$-moments.
\end{enumerate}
\end{Lemma}

\subsection{Law of the Brownian motion, Gaussian kernel}

\pa We fix a dimension $d\in\N^*$ and a finite time horizon $T$. For any $\delta>0$, we denote by $\rho_\delta$ the Gaussian kernel
\[ \rho_\delta(x)=\frac 1{(2\pi\delta)^{\frac{d}{2}}}e^{\frac{-|x|^2}{2\delta}},\; x\in\R^d,\]
and, for any $(t,m)\in[0,T]\times\PR$, and $s\in[t,T]$, by $p^{t,m}_s$ the law of the $d$-dimensional Brownian motion at time $s$, starting at time $t$ with law $m$ : 
\begin{equation}
\label{dens}
p^{t,m}_t=m \mbox{ and, for }s\in(t,T], p^{t,m}_s(dx):=\rho^{t,m}_s(x)dx=\rho_{s-t}*m(x)dx.
\end{equation} 
It is well known that $(\rho_s^{t,m})_{s\in(t,T]}$ satisfies
the heat equation
\begin{equation}
\label{heat}
 \partial_s \rho^{t,m}_s-\frac{1}{2}\Delta \rho^{t,m}_s=0, s\in(t,T].
 \end{equation}
Recall also the Markov property: for all $t\leq t_1\leq t_2\leq T$,
\begin{equation}
\label{Markov}
 p^{t,m}_{t_2}=p^{t_1,p^{t,m}_{t_1}}_{t_2}.
 \end{equation}
 
 \pa In the sequel, we shall repeatedly use the following estimations :
 
 \begin{Lemma} 
\label{dp}
\begin{enumerate}
\item For all $0\leq t\leq s\leq T$ and $m,m'\in\PR_1$, it holds that
 \begin{equation}
 \label{compd}
 d_1(p^{t,m}_s,p^{t,m'}_s)\leq d_1(m,m')\leq d_1(p^{t,m}_s,p^{t,m'}_s)+2\sqrt{d(s-t)}.
 \end{equation}
\item More generally, for all $t\in[0,T]$ and all $k\geq 1$, the map $(s,m)\mapsto p^{t,m}_s$ from $([t,T]\times\PR_k, \lambda|_{[t,T]}\otimes d_k)$ to $(\PR_k,d_k)$ is continuous: The following estimation holds
\begin{equation}
\label{pcont}
d_k(p^{t,m}_s,p^{t,m'}_{s'})^k\leq C_k\left(d_k(m,m')^k+|s-s'|^{k/2}\right),
\end{equation}
where the constant $C_k$ depends only on $k$ and the dimension $d$. 
\end{enumerate}
\end{Lemma}

\begin{proof}
On an arbitrary  probability space $(\bar\Omega,\bar\FR,\bar\P)$,
let  $(X,X')$ be a couple of random variables with marginals $m$ and $m'$ respectively, and  $(B_s)_{s\geq 0}$ a standard $d$-dimensional Brownian motion independent of $(X,X')$. Then the marginals of the couple $(X+B_s,X'+B_{s'})$ are $p^{t,m}_s$
 (resp. $p^{t,m'}_{s'}$) and we can write
\begin{equation}
\label{demdk}
d_k(p^{t,m}_s,p^{t,m'}_{s'})^k\leq \bar\E[|(X+B_s)-(X'+B_{s'})|^k].
\end{equation}
1. For $k=1$, it follows that
\[ d_1(p^{t,m}_s,p^{t,m'}_{s'})\leq \bar\E[|X-X'|]+\sqrt{d|s-s'|},\]
and, since  the couple $(X,X')$ has been chosen arbitrarily,
\begin{equation}
\label{demd1}
d_1(p^{t,m}_s,p^{t,m'}_{s'})\leq d_1(m,m')+\sqrt{d|s-s'|}.
\end{equation}
The relation \eqref{compd} follows using the estimation \eqref{demd1} for $s'=s$ and $s'=t$ and the triangular relation
\[ d_1(m,m')\leq d_1(m,p^{t,m}_s)+d_1(p^{t,m}_s,p^{t,m'}_s)+d_1(p^{t,m'}_s,m').\]
2. Let us come back to the general case. From \eqref{demdk}, we get
\begin{align*}
d_k(p^{t,m}_s,p^{t,m'}_{s'})^k
\leq & 2^k\left( \bar\E[|X-X'|^k]+\bar\E[|B_s-B_{s'}|^k]\right)\\
\leq &C_k\left(\bar\E[|X-X'|^k]+|s-s'|^{k/2}\right).
\end{align*} 
Again, since the couple $(X,X')$ has been chosen arbitrarily, the result follows.
\end{proof}



\subsection{The Brownian motion and the parameters of the game}

\pa Given the finite time horizon $T$, we fix an initial time $t\in[0,T]$. On the set $\Omega_t:=C([t,T],\R^d)$ we denote as usual by $B_s(\omega)=\omega(s), \omega\in\Omega_t, s\in[t,T]$ the canonical process. The set $\Omega_t$ is endowed with $\F_t$, the $\sigma$-algebra generated by $(B_s)_{s\in[t,T]}$. 
Finally, for a fixed probability measure $m\in\PR$, $\P_{t,m}$ denotes the probability on $\Omega_t$ such that, under $\P_{t,m}$, $(B_s)_{s\in[t,T]}$ is a Brownian motion such that $B_t$ is of law $m$. 

\pa 
Let $U$ and $V$ be two compact metric spaces. 
\pa 
Let $f: [0,T]\times\R^d\times U\times V\rightarrow\R$ be a map, which we suppose to be  bounded, continuous in all its variables, Lipschitz in $(t,x)\in [0,T]\times\R^d$ with constant $Lip(f)$, uniformly in $(u,v)$, i.e.:
\[ \forall (u,v)\in U\times V, \, \forall t,t' \in [0,T],\, \forall x,x' \in \R^d, \; |f(t,x,u,v) - f(t',x',u,v')| \leq Lip(f) (|x-x'|+|t-t'|).\] 
Set  $C=\max\{ Lip(f),\|f\|_\infty\}$.
\pa 
As already explained in the introduction, the idea of the game is that, for any fixed initial condition $(t,m)\in[0,T]\times\PR_2$ two players aim to optimize the cost function
$\E_{t,m}\left[\int_t^Tf(s,B_s,u_s,v_s)ds\right]$
where player 1 plays $(u_s)_{s\in[t,T]}$ with values in $U$ and tries to minimize the cost function, while player 2 plays $(v_s)_{s\in[t,T]}$ with values in $V$ and tries to maximize it. Player 1 observes in real time the Brownian motion $(B_s)$ and the controls of player 2, while player 2 only observes the controls of player 1. 
\pa
We do not go more into details,  because we won't analyse this continuous time game but rather approximate it by a sequence of discrete time games. Indeed, as in former works (\cite{cr3},\cite{CRRV},\cite{GR}), while the analysis of the upper value function of the continuous game would be close to what is presented here, there is actually no way to handle its  lower value function. In particular we don't know if the continuous time game has a value. By working on  repeated games, we get around this difficulty thanks to the minmax theorem of Von Neumann: it guaranties that a value exists for each finite-step game. The function we call the ``value of the game'' will be the limit of them.

\pa 
{\bf Assumption:} Throughout the paper, we suppose that the following Isaacs assumption holds: 
\pa
For all $(t,m)\in[0,T]\times\PR_2$,
\be \label{isaacs} 
\inf_{u\in U}\sup_{v\in V} \int_{\R^d}f(t,x,u,v)dm(x)
=\sup_{v\in V}\inf_{u\in U} \int_{\R^d}f(t,x,u,v)dm(x),
\ee
and we denote by $H(t,m)$ the common value.

\begin{Remark}
{\rm Under Isaacs assumption, $H(t,m)$ is the value of the infinitesimal game.
It is also possible to work without Isaacs assumption. In this case $H(t,m)$ has to be replaced by
\[\bar H(t,m)=\inf_{\sigma \in \Delta(U)}\sup_{\tau \in \Delta(V)} \int_U \int_V (\int_{\R^d}f(t,x,u,v)dm(x))d\sigma(u) d\tau(v),\; (t,m)\in[0,T]\times\PR_2,\]
with $\Delta(U)$ (resp. $\Delta(V)$) the set of probability measures on $U$ (resp. $V$). Remark that, since, by assumption, $f$ is continuous in all its variables and $U,V$ are compacts, the infimum and supremum commute in the above definition.
}
\end{Remark}

\begin{Lemma}
\label{HLipsch}
The function $H$ is continuous, bounded and Lipschitz in $t$ and $m\in\PR_2$ (with respect to the $d_1$-distance), with constant $C$. 
\end{Lemma}

\begin{proof} The Lipschitz regularity of $H$  with respect to $t$ follows classically from the Lipschitz regularity in $t$ of $f$.
To prove that $H$ is Lipschitz in $m$, let's write the dual characterization of the Wasserstein distance:
For all $m,m'\in\PR_2$, for all $(u,v)\in U\times V$, it holds that
\[ |\int f(t,x,u,v)dm(x)-\int f(t,x,u,v)dm'(dx)|\leq C\sup_{\varphi \; 1-\mbox{Lipschitz}}\left( \int \varphi dm-\int \varphi dm'\right)=Cd_1(m,m').\]
It follows that the same inequality holds for $H$.
\end{proof}

\subsection{The discrete time game}

\pa Fix $t\in [0,T]$. To each partition $\pi=\{t=t_1<t_2<...<t_N=T\}$ we can associate, as a discretization along $\pi$ of the continuous time game, the following stochastic game $\Gamma_\pi(t,m)$:

\begin{itemize}
\item The variable $B_{t_q}$ is observed by player $1$ before stage $q$ for $q=1,...,N-1$.
\item At each stage, both players choose simultaneously a pair of controls $(u_q,v_q)\in U \times V$.
\item Chosen actions are observed after each stage. 
\item Stage payoff of player $1$ equals $f(t_q, B_{t_q},u_q,v_q)$ (realized stage payoffs are not observed). 
\item The total expected payoff of player $1$ is
\[ \mathbb{E}_{t,m}\left[ \sum_{q =1}^{N-1} (t_{q+1}-t_q) f(t_q,B_{t_q},u_q,v_q) \right].\]
\end{itemize}
The description of the game is common knowledge and we consider the game played in behaviour strategies: at round $q$, player $1$ and player $2$ select simultaneously and independently an action $u_q\in U$ for player $1$ and $v_q\in V$ for player $2$, using some lotteries depending on their past observations. Note that even if the realized stage payoffs are not observed, Player 1 can deduce the value of $f(t_q, B_{t_q},u_q,v_q)$ from his/her observations after stage $q$, which is not possible for Player 2 as he/she does not observe the trajectory of the Brownian motion.
\pa
Formally, a behaviour strategy $\sigma$ for player $1$ is a sequence $(\sigma_{q})_{q=1,...,N-1}$ of transition probabilities:
\[ \sigma_{q} :  ( \R^d \times U \times V )^{q-1}\times  \R^d \rightarrow \Delta(U), \]
where $\sigma_q(B_{t_1},u_1,v_1,...,B_{t_{q-1}},u_{q-1},v_{q-1},B_{t_q})$ denotes the lottery used to select the action $u_q$ played at round $q$ by player $1$, when  past actions played during the game are $(u_1,v_1,...,u_{q-1},v_{q-1})$ and the sequence of observations of player $1$ is $(B_{t_1},...,B_{t_q})$. Let $\Sigma(\pi)$ denote the set of behaviour strategies for player $1$.
Similarly, a behaviour strategy $\tau$ for player $2$  is a sequence $(\tau_{q})_{q=1,...,N-1}$ of transition probabilities depending on  his/her past observations
\[ \tau_{q} : (U\times V )^{q-1} \rightarrow \Delta(V). \]
Let $\TR(\pi)$ denote the set of behaviour strategies for player $2$.
\pa
Let $\P_{t,m,\pi,\sigma,\tau} \in \Delta(C([t,T],\R^d) \times  (U \times V)^{N-1})$ denote the probability on the set of trajectories of $(B_s)$ and actions induced by the strategies $\sigma,\tau$. 
The payoff function in $\Gamma_{\pi}(t,m)$ is defined by
\[ \gamma_{\pi}(t,m,\sigma,\tau):= \mathbb{E}_{t,m,\pi,\sigma,\tau}\left[ \sum_{q=1}^{N-1}  (t_{q+1}-t_q) f(t_q,B_{t_q},u_q,v_q) \right] .\]
It is well known that the game has a value
\begin{align*}
V_{\pi} (t,m) :=& \underset{\tau \in \TR(\pi)}{\sup}\; \underset{\sigma \in  \Sigma(\pi)}{\inf} \; \gamma_{\pi}(t,m,\sigma,\tau)\\
=&\underset{\sigma \in  \Sigma(\pi)}{\inf} \; \underset{\tau \in \TR(\pi)}{\sup}\; \gamma_{\pi}(t,m,\sigma,\tau).  
\end{align*}


\section{An alternative formulation of the value function}
\label{alternative}

We use the notation $m(\phi):=\int \phi dm$.
\pa
On a sufficiently large probability space $(\Omega,\FR,\P)$ we introduce the following set of measure-valued processes:\\
For any fixed $(t,m)\in[0,T]\times\PR_2$, we define $\MR(t,m)$ as the set  of c\`adl\`ag processes $M=(M_s)_{s\in[t,T]}$ with values in the complete separable metric space  $(\PR_1,d_1)$ which satisfy:
\begin{itemize}
\item[$(i)$] For all $\phi \in C_b(\R^d)$, $\E[ M_t(\phi)]=m(\phi)$, 
\item[$(ii)$] for all $t\leq t_1\leq t_2\leq T$, for all $\phi \in C_b(\R^d)$
\[\E\left[ M_{t_2}(\phi)|\FR^M_{t_1}\right]=p_{t_2}^{t_1,M_{t_1}}(\phi),\]
\end{itemize}
where $(\FR^M_s)_{t\leq s\leq T}$ denotes the filtration generated by $M$, completed and made right-continuous, and $C_b(\R^d)$ the set of continuous, bounded functions from $\R^d$ to $\R$.

\begin{Remarks}
\label{remjen}{\rm $\qquad$
\begin{enumerate}
\item These two properties imply:
\be \label{propM3} \forall \phi \in  C_b(\R^d), \forall t' \in [t,T], \;\E[ M_{t'}(\phi)]=p^{t,m}_{t'}(\phi).  \ee
Indeed, we have
\begin{align*}
\E[ M_{t'}(\phi)]&=\E[\E[ M_{t'}(\phi)| \FR^M_{t}]=\E[ p_{t'}^{t,M_t}(\phi)] \\
&= \E[\int_{\R^d} M_t(\phi(x+\cdot)) \rho_{t'-t}(x)dx] = \int_{\R^d}  \E[M_t(\phi(x+\cdot))] \rho_{t'-t}(x)dx \\
&=\int_{\R^d} m(\phi(x+\cdot))\rho_{t'-t}(x)dx =p_{t'}^{t,m}(\phi).
\end{align*}
\item The equalities $(i),(ii)$ as well as \eqref{propM3} extend to any measurable functions with at most quadratic growth using monotone convergence.
\item In particular, \eqref{propM3} implies that the process $M$ takes values in $\PR_2$ almost surely. 
Indeed, for $t'\in [t,T]$ and $\psi(x)=|x|^2$: 
\[ \E[ M_{t'}(\psi)]=p_{t'}^{t,m}(\psi) < \infty,\]
and thus $ M_{t'}(\psi) < \infty$ almost surely.
\item The c\`adl\`ag property (for $d_1$) of the process $M$ and point 4 of Lemma \ref{dro} imply  that  $s\in[t,T]\mapsto M_s(\phi):=\int_{\R^d}\phi(x)dM_s(x)$ is a c\`adl\`ag process for all continuous functions $\phi$ with at most linear growth.
\item Later in the paper, we need following Jensen's-type inequality: \\
for $0\leq t_1\leq t_2\leq t_3\leq T$,
\begin{equation}
\label{jensen}
 d_1(p^{t_2,M_{t_2}}_{t_3},M_{t_1})\leq \E\left[d_1(M_{t_3},M_{t_1})|\FR^M_{t_2}\right].
 \end{equation}
This relation can easily derived from the  dual formulation of the $d_1$-distance, $(ii)$ and Remark 2.:
for all 1-Lipschitz map $\varphi$, 
\[
\int_{\R^d}\varphi d(p^{t_2,M_{t_2}}_{t_3}-M_{t_1})= \E\left[\int_{\R^d}\varphi d(M_{t_3}-M_{t_1})|\FR^M_{t_2}\right]
\leq\E\left[d_1(M_{t_3},M_{t_1})|\FR^M_{t_2}\right].
\]
The result follows by taking the supremum over all 1-Lipschitz maps $\varphi$.

\end{enumerate}
}\end{Remarks}

\pa We set 
\begin{equation}\label{infmeasure}
 V(t,m)=\inf_{(M_s)\in\MR(t,m)}\E\left[\int_t^TH(s,M_s)ds\right]
\end{equation}

\begin{Remark}{\rm  Let
$\D([t,T],\PR_1)$ be the set of c\`adl\`ag trajectories with values in the complete separable metric space $(\PR_1,d_1)$ and denote $(m_s)_{s\in[t,T]}$ the canonical process on $\D([t,T],\PR_1)$. Then an equivalent formulation for $V$ holds:
\begin{equation}\label{infmeas}
 V(t,m)=\inf_{\mu\in\tilde\MR(t,m)}\E_\mu\left[\int_t^TH(s,m_s)ds\right].
\end{equation}
 where $\tilde\MR(t,m)$ is the family of probability measures on $\D([t,T],\PR_1)$ under which the canonical process $(m_s)_{s\in[t,T]}$ satisfies the items $(i)$ and $(ii)$ of the definition of $\MR(t,m)$.}
 \end{Remark}

\pa Let us state the first standard property of $V$.

\begin{Proposition} The function $V$ is convex in $m$.
\end{Proposition}

\begin{proof} 
Fix $t\in[0,T]$ and let $m_1,m_2\in\PR_2$ and $\lambda\in (0,1)$ Set
$m^\lambda=\lambda m_1+(1-\lambda) m_2$. For $\epsilon>0$ let $M^1$ (resp. $M^2$) be $\epsilon$-optimal for $V(t,m_1)$ (resp. for $V(t,m_2)$).
Up to enlarge the probability space, we may assume that there exists $A\in\FR$ such that $P(A)=\lambda$ and $(A,M^1,M^2)$ are mutually independent.
Define now  $M^\lambda$ by
\[ M^\lambda_s=M_s^1\ind_A+M^2_s\ind_{A^c}, \; s\in[t,T]. \]
At first, note that $M^\lambda\in\MR(t,m^\lambda)$. Indeed for all $\phi \in C_b(\R^d)$, we have:
\[ \E[ M_t^{\lambda}(\phi)]= \E[ M_t^{^1}(\phi)\ind_A] +\E[ M_t^{2}(\phi)\ind_{A^c}]= \lambda m_1(\phi)+(1-\lambda)m_2(\phi)=m(\phi).\]
For $t\leq t_1\leq t_2 \leq T$, denoting  $\sigma (\FR^{M^1}_{t_1},\FR^{M^2}_{t_1},A)$ the $\sigma$-field generated by $\FR^{M^1}_{t_1}$,$\FR^{M^2}_{t_1}$ and $A$, we have
\begin{align*}
 \E[ M^{\lambda}_{t_2}(\phi) | \sigma (\FR^{M^1}_{t_1},\FR^{M^2}_{t_1},A) ] &=  \E[ M^{1}_{t_2}(\phi) |\sigma (\FR^{M^1}_{t_1},\FR^{M^2}_{t_1},A)]\ind_A + \E[ M^{2}_{t_2}(\phi) | \sigma (\FR^{M^1}_{t_1},\FR^{M^2}_{t_1},A) ]\ind_{A^c} \\
 &= \E[ M^{1}_{t_2}(\phi) |\FR^{M^1}_{t_1}]\ind_A + \E[ M^{2}_{t_2}(\phi) | \FR^{M^2}_{t_1} ]\ind_{A^c} \\&= p^{t_1,M^1_{t_1}}_{t_2}(\phi)\ind_A+p^{t_1,M^2_{t_1}}_{t_2}(\phi)\ind_{A^c}=p^{t_1,M^{\lambda}_{t_1}}_{t_2}(\phi).
\end{align*}
By taking conditional expectation given $\sigma(M^\lambda_s, s\in[t,t_1]) \subset \sigma (\FR^{M^1}_{t_1},\FR^{M^2}_{t_1},A)$ on both sides, we obtain
\[  \E[ M^{\lambda}_{t_2}(\phi) |\sigma(M^\lambda_s, s\in[t,t_1]) ] =p^{t_1,M^{\lambda}_{t_1}}_{t_2}(\phi).\]
Thus equality being true for all $t\leq t_1\leq t_2 \leq T$, it can be extended to the right-continuous filtration generated by $M^\lambda$ using that the processes $M^{\lambda}(\phi)$ and $p^{\cdot,M^{\lambda}_{\cdot}}_{t_2}(\phi)$ are c\`{a}dl\`{ag} and using the backward martingale convergence theorem. This concludes the proof of property $(ii)$ of the definition of $\MR(t,m)$.
\pa
We deduce that
\[ \begin{array}{rl}
\E\left[\int_t^TH(s,M^\lambda_s)ds\right]=&
\lambda\E\left[\int_t^TH(s,M^1_s)ds\right]+
(1-\lambda)\E\left[\int_t^TH(s,M^2_s)ds\right]\\
\leq &\lambda V(t,m^1)+(1-\lambda)V(t,m^2)+2\epsilon.
\end{array}\]
The result follows by letting $\varepsilon$ go to zero. 
\end{proof}

\pa We state now the main result of this section. It explains why we may consider $V$ as a natural definition of a value for the continuous-time game described in the introduction. 
\begin{Theorem}
\label{WV}
For all $(t,m)\in[0,T]\times\PR_2$, the limit
$\lim_{|\pi| \rightarrow 0} V_\pi(t,m)$ exists and coincides with $V(t,m)$, 
where for a partition $\pi=\{ t=t_1<\ldots <t_N=T\}$, $|\pi|$ denotes its mesh, i.e.:
\[|\pi|=\sup\{ |t_{i+1}-t_i|,\; i=1,\ldots,N-1\}.\]
\end{Theorem}

\pa The proof requires several preliminary results, which are presented in the next subsections.

\subsection{A generalized splitting  lemma}

\pa In order to compare the limit points of $(V_\pi(t,m))_{|\pi|\rightarrow 0}$ and $V(t,m)$, we have to make correspond strategies in $\sigma\in\Sigma(\pi)$ and elements of $\MR(t,m)$. A classical tool for this, the so called ``splitting argument'' is well known for games where the asymmetrical observed parameter has a finite support  (see the elementary version in Aumann-Machler \cite{AumannMaschler} or Sorin \cite{Sorinfirstcourse}, and in \cite{CRRV} its adaptation to a dynamic setting). Here we need a more general version.\\

\pa Let us recall a fundamental result of Blackwell and Dubins.

\begin{Theorem}\label{dubinsapp}(Blackwell-Dubins \cite{blackwelldubins}) \\
Let $E$ be a polish space, $\Delta(E)$ the set of Borel probabilities on $E$, $([0,1],\BR([0,1]),\lambda)$ the unit interval equipped with Lebesgue's measure, and  $U$  the canonical element in $[0,1]$. There exists a measurable mapping
\[ \Phi_E : \Delta(E)\times [0,1]  \longrightarrow E  \]
such that for all $\mu \in \Delta(E)$, the law of $\Phi_E(\mu,U)$ is $\mu$.
\end{Theorem}

\begin{Proposition}\label{splitting}
Let $M$ in $\MR(t,m)$ and $\pi=\{t=t_1<...<t_N=T\}$ a partition of $[t,T]$. 
Given a Brownian motion $(B_s)_{s \in [t,T]}$ with initial law $m$, and $(U_1,...,U_{N-1})$ some independent random variables independent of $(B_s)$, all uniformly distributed on $[0,1]$ and defined on a same  probability space $(\bar\Omega,\bar\FR,\bar\P)$, there exist variables $(M'_{t_q})_{q=1,...,N-1}$ defined on $\bar\Omega$, having the same law as $(M_{t_q})_{q=1,...,N-1}$, and such that for all $q=1,...,N-1$:
\begin{itemize}
\item $M'_{t_q}$ is measurable with respect to $(B_{t_1},U_1,...,B_{t_{q}},U_{q})$.
\item The conditional law of $B_{t_q}$ given $(M'_{t_1},...,M'_{t_q})$ is precisely $M'_{t_q}$. 
\end{itemize}
\end{Proposition}

\begin{proof}
At first, let us assume that $M$ is defined on a probability space $(\Omega,\FR,\P)$ sufficiently large so that there exists a family  $(\tilde U_1,...,\tilde U_{N-1})$ of independent uniform random variables independent of $M$, all defined on $\Omega$. Define the variable $\tilde{B}_{t_1}=\Phi_{\R^d}(M_{t_1},\tilde U_1)$, so that the conditional law of  $\tilde{B}_{t_1}$ given $M_{t_1}$ is almost surely equal to $M_{t_1}$. Let $f_{t_1}(\tilde{B}_{t_1})$ denote a version of the conditional law of $M_{t_1}$ given $\tilde{B}_{t_1}$, meaning that $f_{t_1}$ is a measurable map from $\R^d$ to  $\Delta(\PR_1)$.
Define then on the probability space $\bar\Omega$ the variable $M'_{t_1}=\Phi_{\PR_1}(f_{t_1}(B_{t_1}),U_1)$, so that $(B_{t_1},M'_{t_1})$ has the same law as $(\tilde{B}_{t_1},M_{t_1})$, implying that the conditional law of $B_{t_1}$ given $M'_{t_1}$ is almost surely equal to $M'_{t_1}$.
\pa
We proceed by induction. Let us assume that the variables $(M'_{t_1},....,M'_{t_{q-1}})$ are defined on $\bar\Omega$ for some $1\leq q \leq N-2$, and have the required properties.
Define $\tilde{B}_{t_q}=\Phi_{\R^d}(M_{t_q},\tilde U_q)$, so that the conditional law of  $\tilde{B}_{t_q}$ given $(M_{t_1},....,M_{t_q})$ is almost surely equal to $M_{t_q}$. Let $f_{t_q}(\tilde{B}_{t_q},M_{t_1},...,M_{t_{q-1}})$ denote a version of the conditional law of $M_{t_q}$ given $(\tilde{B}_{t_q},M_{t_1},...,M_{t_{q-1}})$, meaning that $f_{t_q}$ is a measurable map from $\R^d \times (\PR_1)^{q-1}$ to $\Delta(\PR_1)$.
Define then, on the probability space $\bar\Omega$, the variable $M'_{t_q}=\Phi_{\PR_1}(f_{t_q}(B_{t_q},M'_{t_1},....,M'_{t_{q-1}}),U_q)$, so that $(B_{t_q},M'_{t_1},...,M'_{t_q})$ has the same law as $(\tilde{B}_{t_q},M_{t_1},...,M_{t_q})$, implying that the conditional law of $B_{t_q}$ given $(M'_{t_1},...,M'_{t_q})$ is almost surely equal to $M'_{t_q}$.
\end{proof}

\subsection{Compactness and continuity}\label{compactness}

\pa 
The next step is to show that, up to enlarge the underlying probability space,  the infimum in the formulation \eqref{infmeasure} is attained for some law in $\tilde{\MR}(t,m)$. This statement is equivalent to claim that the space of probability measures $\tilde\MR(t,m)$ defined in \eqref{infmeas} is compact for the appropriate topology.
\pa 
For that, we endow the set $\D([t,T],\PR_1)$ with the topology $\kappa_1$ of convergence in measure together with the convergence of the value at time $T$.
It means that a sequence $h^n$ converges to $h$ if 
\[ \forall \varepsilon>0, \; \int_{[t,T]} \ind_{d_1(h^n(s),h(s))>\varepsilon} ds \underset{ n\rightarrow \infty}{\longrightarrow} 0, \; d_1(h^n(T),h(T))\underset{ n\rightarrow \infty}{\longrightarrow} 0.\]
This topology is metrizable, and makes the space $(\D([t,T],\PR_1),\kappa_1)$ a separable metric space which is not topologically complete.
Let $\LR(\kappa_1)$ denote the associated convergence in law on $\Delta(\D([t,T],\PR_1))$.


\pa The following lemma is a corollary of the main result in Kurtz \cite{Kurtz} which extends the classical result of Meyer-Zheng \cite{MeyerZheng} to measure-valued processes:  \\

\begin{Lemma}[Consequence of Theorem 1.1  and Corollary 1.4 in \cite{Kurtz}]\label{tension}
Consider a sequence of processes $(M^n)_{n\geq 0}\subset\MR(t,m)$ for $m\in \PR_2$.
Then the set of laws of the processes $(M^n)_{n \geq 0}$ is $\LR(\kappa_1)$-relatively compact. \end{Lemma}

\begin{proof}
In order to exactly fit in the framework of \cite{Kurtz}, one may extend the definition of the processes to $(M^n)_{s \in [t,+\infty)}$ by $M^n_s=M^n_T$ for $s\geq T$.
We have to verify  the condition $C1.1(i)$ (compactness containment) of Theorem 1.1 in \cite{Kurtz}, i.e. that for all $\varepsilon>0$, there exists a compact $K \subset \PR_1$ such that 
\[ \liminf_{n \rightarrow \infty} \P(M^n_s \in K, \forall s\in [t,T]) \geq 1-\varepsilon. \]
Let $\psi(x)=|x|^2$. Note that, using the properties of the Gaussian kernel, for all $s'\geq s$,
\[  \E[ M^n_{s'}(\psi)| \FR_{s}] =p^{s,M^n_s}_{s'}(\psi)= M^n_s(\psi) +d(s'-s).\]
This implies that $M^n(\psi)$ is a non-negative submartingale.\\
Using Doob's inequality, we have therefore
\[ \P ( \sup_{s\in [t,T]} M^n_s(\psi) > C ) \leq \frac{\E[ M^n_T(\psi)]}{C} = \frac{m(\psi)+d(T-t)}{C}.\]
As the set $\{ \mu \in \PR_1 \,|\, \mu(\psi) \leq C \}$ is compact in $(\PR_1,d_1)$ (see Lemma \ref{dro}), the condition holds by choosing $C \geq \frac{m(\psi)+d(T-t)}{\varepsilon}$.
\pa
We verify now the condition of Corollary 1.4. in \cite{Kurtz}, i.e. that there exists a countable separating subset $\{f_i\}_{i\in I}$ in $C_b(\PR_1)$ such that for all $i \in I$ 
\begin{equation}
\label{gnu}
\sup_n \sup_{0=t_0<...<t_k<...<t_N=T} \E\left[ \sum_{k=0}^{N-1} \left| \E[f_i(M^n_{t_{k+1}}) - f_i(M^n_{t_k}) | \FR^{M^n}_{t_k}] \right|\right] <+\infty.
\end{equation}
For this, we can choose a sequence $(\phi_i)$ in $ C^\infty(\R^d)$ which separates points in $\PR_1$ (for example, one may choose the maps $\cos(\langle t_j, \cdot\rangle )$ and $\sin(\langle t_j, \cdot\rangle )$ when $t_j$ varies in a countable dense subset of $\R^d$), so that the maps $f_i:\mu \rightarrow \int \phi_i d\mu$ are separating in $C_b(\PR_1)$.
Moreover for these maps, the quantity \eqref{gnu} is bounded by $(T-t)\|D^2\phi_i\|_\infty$ since for all $i,k$, It\^o's formula implies:
\[ |\E[M^n_{t_{k+1}}(\phi_i) - M^n_{t_k}(\phi_i) | \FR^{M^n}_{t_k}]|=|p^{t_k,M^n_{t_k}}_{t_{k+1}}(\phi_i) -M^n_{t_k}(\phi_i) | \leq \frac{d}{2}\|D^2\phi_i\|_\infty(t_{k+1}-t_k).\]

\pa The result of Kurtz implies (see point (b) of Theorem 1.1 having in mind that Lebesgue-almost sure convergence of trajectories implies convergence in measure) that there exist a process $M$ and subsequence of $M^n$ which $\LR(\kappa_1)$-converges to  $M$. The result follows. 
\end{proof}

\pa A similar result for martingales can be found in Cardaliaguet-Rainer \cite{cr3} with a proof based on the method of Meyer and Zheng \cite{MeyerZheng}.

\begin{Proposition}
\label{Mcompact}
$\tilde\MR(t,m)$ is $\LR(\kappa_1)$  compact in $\Delta(\D([t,T],\PR_1))$. 
\end{Proposition}
\begin{proof}
Using the preceding lemma \ref{tension}, it remains to prove that $\tilde\MR(t,m)$ is closed  in $\Delta(\D([t,T],\PR_1))$.
Let $(\mu^n)$ be a sequence  in $\tilde\MR(t,m)$ that converges for $\LR(\kappa_1)$ to some law $\mu$  in $\Delta(\D([t,T],\PR_1))$. Using Skorokhod representation Theorem for separable metric spaces (see Theorem 11.7.31 in \cite{dudley}), we can find on some probability space $(\Omega,\FR,\P)$ a sequence of processes  $(M^n)_{n\geq 1} $ of law $\mu^n$ and a process $M=(M_s)_{s\in [t,T]}$ of law $\mu$ such that $(M^n)\overset{\kappa_1}{\rightarrow} M$ almost surely. Up to extracting a subsequence, we can also assume that there exists a subset $R$ of full Lebesgue measure in $[t,T]$ and containing $T$ such that for all $s \in R$,  $M^n_s \overset{d_1}{\rightarrow} M_s$ almost surely. 
\pa
Let $\phi \in C_b(\R^d)$.  Then, for all $s\in R$, we have almost surely
\[  M^n_s(\phi) \rightarrow M_s(\phi).\]
Given some bounded continuous maps $f_1,...,f_k: \PR_1 \rightarrow \R$ and $t_1,...,t_k \leq s \leq s'$ in $R$, we have
\[ \forall n \geq 0, \; \E[ f_1(M^n_{t_1})...f_k(M^n_{t_k}) M^n_{s'}(\phi)] = \E[ f_1(M^n_{t_1})...f_k(M^n_{t_k}) p^{s,M^n_s}_{s'}(\phi)].\]
Taking the limit as $n$ goes to $+\infty$, we deduce that 
\[  \E[ f_1(M_{t_1})...f_k(M_{t_k}) M_{s'}(\phi)] = \E[ f_1(M_{t_1})...f_k(M_{t_k}) p^{s,M_s}_{s'}(\phi)],\]
since the map $m' \rightarrow p^{s,m'}_{s'}$ is $d_1$-continuous.
Recall now that, by definition, $M$ is c\`{a}dl\`{a}g. This implies that $\sigma(M_r, r\in [t,s])=\sigma(M_r, r\in R\cap[t,s])$.
It follows that
\[  \E[ M_{s'}(\phi_i) | \sigma(M_r, r\in [t,s])] =  p^{s,M_s}_{s'}(\phi_i).\]
This property holds true almost surely for all $s,s'$ in a countable dense subset of $[t,T]$ containing $T$. It can therefore be extended to all $s,s' \in [t,T]$ when replacing $ \sigma(M_r, r\in [t,s])$ by its right-continuous augmentation $\FR^M_s$.
\pa
Similarly, for all $n\geq 0$ and all $s\in R$, we have
\[ \E[ M^n_s(\phi)]=p^{t,m}_s(\phi). \]
We deduce that $\E[ M_s(\phi)]=p^{t,m}_s(\phi)$ by taking the limit, and the property extends to all $s\in [t,T]$ by right-continuity. This also implies that  $\E[M_t(\phi)]=m(\phi)$. It follows that $\MR(t,m)$ is closed.
This fact, together with Theorem \ref{tension} gives the compactness of $\tilde{\MR}(t,m)$.
\end{proof}

\subsection{$d_1$-variation of belief processes.}

\begin{Lemma}\label{d1var}
Given any partitions $\pi=\{t=t_1<....<t_N=T\}$, we have
\[ \lim_{|\pi|\rightarrow 0} \sup_{M \in \MR(t,m) } \E[\sum_{q=1}^{N-1} (t_{q+1}-t_q) d_1(M_{t_q},\hat{M}_{t_{q}})] =0,     \]
\[ \lim_{|\pi|\rightarrow 0} \sup_{M \in \MR(t,m) } \E[\sum_{q=1}^{N-1} (t_{q+1}-t_q) d_1(M_{t_{q+1}},\hat{M}_{t_{q+1}})] =0,     \]
where we have set $\hat{M}_{t_1}=m$ and for $q=2,...,N$, $\hat{M}_{t_{q}}=p^{t_{q-1},M_{t_{q-1}}}_{t_q}$.
\end{Lemma}

\begin{proof}
We only prove the first assertion as the proof of the second is completely similar.
\pa
Let $\LR^*$ denote the set of $1$-Lipschitz functions on $\R^d$ taking value $0$ at $0$. 
Define the norm $\|f\|_* := \sup_{x \in \R^d} \frac{|f(x)|}{1+|x|^2}$ on the vector space of continuous functions with at most quadratic growth. The space $\LR^*$ is compact for the norm $\|.\|_*$: for all $\varepsilon>0$, there exists a finite family $\{f_i, i=1,....,N_\varepsilon\}$ such that, for all $f \in \LR^*$, there exists $f_i$ such that $\|f-f_i\|_* \leq \varepsilon$.
Let $(U_i)_{i=1,...,N_\varepsilon}$ denote a measurable partition of $\LR^*$ such that 
\[ \forall f \in U_i, \|f-f_i\|_* \leq \varepsilon .\]
Note that for all $f\in \LR^*$, $\E[ M_{t_q}(f) | \FR^{M}_{t_{q-1}} ]= \hat{M}_{t_q}(f)$ for $q=2,...,N-1$ and $\E[M_{t_1}(f)]= \hat{M}_{t_1}(f)$. 
For all $\mu,\mu'\in\PR_1$, we have by definition
\[ d_1(\mu,\mu')= \sup_{f \in \LR^*} \mu(f) - \mu'(f).\]
If we endow $\PR_2$ with the distance $d_2$ (which induces the same Borel structure as $d_1$) and $\LR^*$ with the norm $\|.\|_*$, then,  using point 4 of Lemma \ref{dro}, one easily checks that the map
\[ (\mu,\mu',f) \in \PR_2\times \PR_2 \times \LR^* \rightarrow  \mu(f) - \mu'(f),\]
is jointly continuous. Applying Proposition 7.33 in Bertsekas-Shreves \cite{BS}, there exists a measurable selection $\Phi:  \PR_2\times \PR_2 \rightarrow \LR^*$ such that
\[ \forall \mu,\mu' \in \PR_2, \; d_1(\mu,\mu')= \int \Phi(\mu,\mu') d\mu - \int \Phi(\mu,\mu') d\mu' .\]
\pa
Define the random functions $x\mapsto g_q(x)= \Phi(M_{t_q},\hat{M}_{t_{q}})(x)\ind_{(M_{t_q},\hat{M}_{t_{q}}) \in \PR_2\times \PR_2}$ so that
\[ \E[\sum_{q=1}^{N-1} (t_{q+1}-t_q) d_1(M_{t_q},\hat{M}_{t_{q}})]=\E[\sum_{q=1}^{N-1} (t_{q+1}-t_q)( M_{t_q}(g_q)-\hat{M}_{t_{q}}(g_q))].\]
We have, for all $q$,
\[ M_{t_q}(g_q)= M_{t_q}(\sum_i \ind_{g_q\in U_i} f_i) + M_{t_q}(\sum_i \ind_{g_q\in U_i} (g_q-f_i)) .\]
Using that, for $\P$-almost all $\omega\in\Omega$, 
\[ \forall x \in \R^d, \; \sum_i \ind_{g_q\in U_i} |g_q(\omega,x)-f_i(x)| \leq \varepsilon (1+|x|^2) ,\]
we deduce that
\[ |  M_{t_q}(g_q)- M_{t_q}(\sum_i \ind_{g_q\in U_i} f_i)| \leq \varepsilon \int (1+|x|^2) dM_{t_q}(x) .\]
Finally, this implies that
\[ \E[ |  M_{t_q}(g_q)- M_{t_q}(\sum_i \ind_{g_q\in U_i} f_i)|] \leq \varepsilon \int (1+|x|^2) dp^{t,m}_{t_q}(x) \leq \varepsilon(1+C_1),\]
for some constant $C_1$ depending on $m$ and $(T-t)$.
The same argument leads to
 \[ \E[ |  \hat{M}_{t_q}(g_q)- \hat{M}_{t_q}(\sum_i \ind_{g_q\in U_i} f_i)|] \leq \varepsilon \int (1+|x|^2) dp^{t,m}_{t_q}(x) \leq \varepsilon(1+C_1).\]
We obtain:
\begin{equation}
\begin{array}{l}
\label{d11}
\E[\sum_{q=1}^{N-1} (t_{q+1}-t_q)( M_{t_q}(g_q)-\hat{M}_{t_{q}}(g_q))]\\
 \qquad\qquad\qquad\leq 2T\varepsilon(1+C_1) + \E[\sum_{q=1}^{N-1} (t_{q+1}-t_q) \sum_{i=1}^{N_\varepsilon}\ind_{g_q \in U_i}( M_{t_q}(f_i)-\hat{M}_{t_{q}}(f_i))]\\
\qquad\qquad\qquad\leq 
2T\varepsilon(1+C_1) +  \sum_{i=1}^{N_\varepsilon} \E[\sum_{q=1}^{N-1} (t_{q+1}-t_q) | M_{t_q}(f_i)-\hat{M}_{t_{q}}(f_i)|]
\end{array}
\end{equation}
Note that the second inequality is far from being precise, since we just have bounded the indicator functions by $1$. Its advantage is that the integrands are now all deterministic.
\pa
Now, for $\eta>0$ and $i=1,...,N_\varepsilon$, we define the mollification $f_i^\eta = f_i * \rho_\eta$ (recall that $\rho_\eta$ denotes the Gaussian kernel, see section 2.2).  Using that $f_i$ is $1$-Lipschitz, it is then well-known that $f_i^\eta$ is smooth and that
\[\|f_i-f_i^\eta\|_\infty \leq C_2 \sqrt{\eta}, \; \|\nabla f_i^\eta \|_\infty\leq 1, \; \|D^2 f_i^\eta \|_\infty \leq \frac{C_2}{\sqrt{\eta}},  \]
with  $C_2 =\int_{\R^d} |x| \rho_1(x)dx$. We deduce that, for all $i=1,...,N_\varepsilon$:
\begin{equation}
\label{D12}
\begin{array}{rl}
\E[\sum_{q=1}^{N-1} (t_{q+1}-t_q)| M_{t_q}(f_i)-\hat{M}_{t_{q}}(f_i)|] & \leq \E[\sum_{q=1}^{N-1} (t_{q+1}-t_q)| M_{t_q}(f_i^\eta)-\hat{M}_{t_{q}}(f_i^\eta)|]+2C_2T\sqrt{\eta}
\end{array}
\end{equation}
Now, by Cauchy-Schwarz, we have 
\begin{equation}
\label{D13}
\begin{array}{rl}
\E[\sum_{q=1}^{N-1} (t_{q+1}-t_q)| M_{t_q}(f_i^\eta)-\hat{M}_{t_{q}}(f_i^\eta)|] & \leq \sqrt{|\pi|} \E[\sum_{q=1}^{N-1} \sqrt{t_{q+1}-t_q}| M_{t_q}(f_i^\eta)-\hat{M}_{t_{q}}(f_i^\eta)|]\\
&\leq \sqrt{|\pi|}\sqrt{T}  \left( \E \left[\sum_{q=1}^{N-1} \Big( M_{t_q}(f_i^\eta)-\hat{M}_{t_{q}}(f_i^\eta)\Big)^2 \right]\right)^{1/2} .
\end{array}
\end{equation}
Then, using that $(M_{t_q}(f_i^\eta)-\hat{M}_{t_{q}}(f_i^\eta))$ is a sequence of martingale increments, we have
\begin{align*}
 \E\left[\sum_{q=1}^{N-1} \Big( M_{t_q}(f_i^\eta)-\hat{M}_{t_{q}}(f_i^\eta) \Big)^2 \right] &=\E \left[ \left( \sum_{q=1}^{N-1}  \Big(M_{t_q}(f_i^\eta)-\hat{M}_{t_{q}}(f_i^\eta) \Big) \right)^2 \right]\\
 &=\E\left[ \left(M_{t_{N-1}}(f_i^\eta)-m(f_i^\eta)+ \sum_{q=1}^{N-2}  \Big(M_{t_q}(f_i^\eta)-\hat{M}_{t_{q+1}}(f_i^\eta) \Big) \right)^2 \right]
 \end{align*}
Therefore it holds that
\begin{align*}
\left( \E\left[\sum_{q=1}^{N-1} \Big( M_{t_q}(f_i^\eta)-\hat{M}_{t_{q}}(f_i^\eta) \Big)^2 \right]\right)^{1/2} \leq& m(f_i^\eta) + \E\left[ \Big(M_{t_{N-1}}(f_i^\eta)\Big)^2\right]^{1/2} \\
&+  \left( \E \left[ \left(\sum_{q=1}^{N-2} \Big( M_{t_q}(f_i^\eta)-\hat{M}_{t_{q+1}}(f_i^\eta) \Big) \right)^2 \right]\right)^{1/2}.
\end{align*}
Since $f_i^\eta$ is smooth, using It\^o's formula, for all $\mu \in \PR_2$ and all $t\leq s_1 \leq s_2 \leq T$, we have:
\[ p^{s_1,\mu}_{s_2}(f_i^{\eta})=\E_{s_1,\mu}[ f_i^{\eta}(B_{s_2}) ] = \E_{s_1,\mu}\left[ f_i^{\eta}(B_{s_1})+ \int_{s_1}^{s_2}  \nabla f_i^{\eta}(B_{r}) \cdot dB_r +\frac{1}{2} \int_{s_1}^{s_2} \tr(D^2f_i^{\eta}(B_r))dr \right].\]
Since $\| \nabla f_i^{\eta} \|_\infty \leq 1$, the stochastic integral is a martingale, and thus 
\[  | p^{s_1,\mu}_{s_2}(f_i^{\eta})- \mu(f_i^\eta) | =  \left| \E_{s_1,\mu}\left[\frac{1}{2} \int_{s_1}^{s_2} \tr(D^2f_i^{\eta}(B_r))dr \right] \right| \leq \frac{dC_2}{2\sqrt{\eta}} (s_2-s_1).\]
This formula holds in particular for $\mu=M_{t_q}, q=1,...,N-2$:
\[ |M_{t_q}(f_i^\eta)-\hat{M}_{t_{q+1}}(f_i^\eta)|= | M_{t_q}(f_i^\eta) -p^{t_q,M_{t_q}}_{t_{q+1}}(f_i^\eta)| \leq \frac{dC_2}{2\sqrt{\eta}} (t_{q+1}-t_q).\] 
This leads to the following estimation:
\begin{align*}
\left( \E\left[ \left(\sum_{q=1}^{N-2}  M_{t_q}(f_i^\eta)-\hat{M}_{t_{q+1}}(f_i^\eta)\right)^2 \right]\right)^{1/2}
 \leq  \frac{dC_2T}{2\sqrt{\eta}}.
\end{align*}
On the other hand, Jensen's inequality implies
\[ M_{t_{N-1}}(f_i)^2 \leq M_{t_{N-1}}(f_i^2),\]
which in turn implies (with $\psi(x)=|x|^2$, recall that $f_i(0)=0$ so that $|f_i|^2\leq \psi$))
\[ \E[M_{t_{N-1}}(f_i^\eta)^2]^{1/2} \leq  C_2\sqrt{\eta}+ \E[M_{t_{N-1}}(f_i^2)]^{1/2} \leq C_2\sqrt{\eta}+\E[ M_{t_{N-1}}(\psi)]^{1/2}= C_2\sqrt{\eta}+\sqrt{m(\psi)+(T-t)}.\]
Summing up, we proved that
\begin{equation}
\label{d14}
 \left( \E\left[\sum_{q=1}^{N-1} \Big( M_{t_q}(f_i)-\hat{M}_{t_{q}}(f_i) \Big)^2 \right]\right)^{1/2}\leq  m(f_i) +\sqrt{m(\psi)+(T-t)} +\frac{d C_2 T}{2\sqrt{\eta}}+4C_2\sqrt{\eta} .
 \end{equation}
Resuming \eqref{d11}-\eqref{d14}, we get finally
\begin{align*}
\E&\left[\sum_{q=1}^{N-1} (t_{q+1}-t_q)( M_{t_q}(g_q)-\hat{M}_{t_{q}}(g_q))\right]\\
 &\leq \;2T\varepsilon(1+C_1)
 + N_\varepsilon \sqrt{|\pi|T}\left( |m|_2 +\sqrt{m(\psi)+(T-t)} +\frac{d C_2T}{2\sqrt{\eta}}+4C_2\sqrt{\eta}\right)  +2N_\varepsilon C_2 T\sqrt{\eta}.
\end{align*} 
This implies that for all $\varepsilon>0$ and $\eta>0$,
\[ \limsup_{|\pi| \rightarrow 0}  \sup_{M \in \MR(t,m) }  \E\left[\sum_{q=1}^{N-1} (t_{q+1}-t_q)d_1(M_{t_q},\hat{M}_{t_{q}})\right]\leq 2T\varepsilon(1+C_1)+2N_\varepsilon C_2 T\sqrt{\eta} ,\]
and the result follows by sending $\eta$ and then $\varepsilon$ to zero.
\end{proof}
\pa
We now state a corollary of the previous Lemma.

\begin{Lemma}\label{Hriemann}
Let $M^n$ be a $\LR(\kappa_1)$-convergent sequence in $\MR(t,m)$ with limit $M$. Then for any sequence of partitions $\pi^n=\{t=t^n_1<...<t^n_{N_n}=T\}$ with $|\pi^n| \rightarrow 0$, we have
\[ \E\left[\sum_{q=1}^{N_n-1} (t^n_{q+1}-t^n_q) H(t^n_q,M^n_{t^n_q})\right] \rightarrow \E\left[\int_t^T H(s,M_s)ds\right].\] 
\end{Lemma}
\begin{proof}
At first note that $\P \rightarrow \E_{\P}[\int_t^T H(s,M_s)ds]$ is $\LR(\kappa_1)$-continuous by construction since $x(s) \in \D([t,T],\PR_1)\rightarrow \int_t^T H(s,x(s)) ds$ is $\kappa_1$-continuous. It follows that: 
\begin{equation}
 \label{convloi} 
 \E\left[\int_t^T H(s,M^n_s)ds \right]\rightarrow \E\left[\int_t^T H(s,M_s)ds\right].
 \end{equation}
On the other hand, since, by Lemma \ref{HLipsch}, $H$ is $C$-Lipschitz, we have,  for all $n$ and all $q$, 
\begin{equation}
\label{HnH}
\bigg|(t^n_{q+1}-t^n_q) H(t^n_q,M^n_{t^n_q}) -\int_{t^n_q}^{t^n_{q+1}} H(s,M^n_s)ds \bigg|\leq 
C  \int_{t^n_q}^{t^n_{q+1}} ((s-t^n_q ) + d_1(M^n_s,M^n_{t^n_q} ) )ds 
\end{equation}
Using that, by Lemma \ref{dp},  it holds for all $s\in[t_q,t_{q+1}]$
\[
 d_1(M^n_s,M^n_{t^n_q})\leq d_1(p^{s,M^n_s}_{t^n_{q+1}},M^n_{t^n_q})+\sqrt{d(t^n_{q+1}-s)},\]
we get
\[\begin{array}{l}
\bigg|(t^n_{q+1}-t^n_q) H(t^n_q,M^n_{t^n_q}) -\int_{t^n_q}^{t^n_{q+1}} H(s,M^n_s)ds \bigg|\\
\qquad\qquad\qquad\qquad\leq  C\left(\int_{t^n_q}^{t^n_{q+1}}d_1(p_{t^n_{q+1}}^{s,M^n_s},M^n_{t^n_q})ds+(t^n_{q+1}-t^n_q)(|\pi^n|+\sqrt{d|\pi^n|}\right).
 \end{array}
 \]
Now remark that for all $\mu \in \PR_1$, the map $\nu \in \PR_1 \rightarrow d_1(\mu,\nu)$ is convex. Therefore, taking conditional expectations given $\FR^{M^n}_s$  and using the notation $\hat M^n_{t^n_{q+1}}=p^{t_q,M^n_{t_q}}_{t_{q+1}}$, it follows from {\blu Remark \ref{remjen} 5.} that:
\begin{align*}
\E[d_1(p^{s,M^n_s}_{t^n_{q+1}},M^n_{t^n_q})] \leq & \E[  d_1(M^n_{t^n_{q+1}},M^n_{t^n_q} )]\\
\leq & \E[d_1(M^n_{t^n_{q+1}},\hat M^n_{t^n_{q+1}})]+\sqrt{d(t^n_{q+1}-t^n_q)}.
\end{align*} 
We deduce that:
\[
\E[|(t^n_{q+1}-t^n_q) H(t^n_q,M^n_{t^n_q}) - \int_{t^n_q}^{t^n_{q+1}} H(s,M^n_s)ds |]
\leq C (t^n_{q+1}-t^n_q)\left( |\pi^n|+\E[d_1(M^n_{t^n_{q+1}},\hat M^n_{t^n_{q+1}})]+2\sqrt{d|\pi^n|}\right)\]
and finally
\begin{align*}
|\E[\sum_{q=1}^{N_n-1} (t^n_{q+1}-t^n_q) H(t^n_q,M^n_{t^n_q})]-\E[\int_t^T &H(s,M_s)ds]| 
\leq \left|\E\int_t^TH(s,M^n_s)ds-\E[\int_t^TH(s,M_s)ds\right|\\
&+CT \left(\E\sum_{q=1}^{N_n-1}(t^n_{q+1}-t^n_q)d_1(M^n_{t^n_{q+1}},\hat M^n_{t^n_{q+1}})]+|\pi^n|+2\sqrt{d|\pi^n|}\right).
\end{align*}
The result follows by \eqref{convloi} and Lemma \ref{d1var}.
\end{proof}

\subsection{Proof of the alternative formulation}

\pa We are now ready to establish Theorem \ref{WV} : $\lim_{|\pi|\rightarrow 0}V^\pi$ exists and is equal to the value of the martingale-optimization problem $V$, as a consequence of the two propositions \ref{limsup} and \ref{liminf} below.

\begin{Proposition}
\label{limsup} 
\begin{equation}\label{ineqlimsup}
\limsup_{|\pi| \rightarrow 0} V_{\pi}(t,m) \leq V(t,m).
\end{equation}
\end{Proposition}

\begin{proof} We denote by $(\bar\Omega,\bar\FR,\bar\P)$ an extension of the canonical Wiener space $(\Omega_t,\FR_t,\P_{t,m})$ which supports also a family $(U_q,V_q)_{q=1,...,N-1}$ of independent random variables with uniform law on $[0,1]$, independent from $(B_s)_{s \in [t,T]}$. \\
Consider an arbitrary element $M$ in $\MR(t,m)$ and a partition $\pi=\{ t=t_1<...<t_N=T\}$. Using the splitting proposition \ref{splitting}, we can define some sequence of measure valued random variables on $\bar\Omega$, $(M'_{t_q})_{q=1,...,N-1}$, with same law as $(M_{t_q})_{q=1,...,N-1}$, having the following properties 
\begin{itemize}
\item[$a)$] $M'_{t_q}$ is measurable with respect to $(B_{t_1},U_1,...,B_{t_{q}},U_{q})$.
\item[$b)$] The conditional law of $B_{t_q}$ given $(M'_{t_1},...,M'_{t_q})$ is precisely $M'_{t_q}$. 
\end{itemize}
Let $(s,m)\mapsto u^*(s,m)$ be a measurable selection of $\mbox{Argmin}_{u\in U}\max_{v\in V}\int_{\R^d}f(s,x,u,v)dm(x)$ (which exists by Proposition 7.33 in \cite{BS}). Note that we have for all $(s,m)\in [t,T]\times \PR_2$:
\be \label{ineqselec} \forall v\in V, \; \int_{\R^d} f(s,x,u^*(s,m),v)dm(x) \leq H(s,m).\ee 
With these ingredients, we shall compose a strategy $\sigma^*$ for player 1: at each step $q$, the action of player $1$ is given by
\[ u'_q=u^*(t_q,M'_{t_q}).\]
Thanks to property $a)$, this definition induces a behavior strategy $\sigma^*\in \Sigma(\pi)$ that does not depend on player $2$'s actions, where $\sigma^*_q(B_{t_1},\ldots,B_{t_q},u_1,\ldots,u_{q-1})$ is simply a version of the conditional law of $u'_q$ given $(B_{t_1},\ldots,B_{t_q},u'_1,\ldots,u'_{q-1})$.
\pa
Let player 2 chose some arbitrary strategy $\tau\in\TR(\pi)$. 
Without loss of generality, we can compute the payoff associated to the strategies $\sigma^*$ and $\tau$ on the probability space $\bar \Omega$. Precisely, using the notations of Theorem \ref{dubinsapp}, we define the actions of player 2 by:
\[ v'_q=\Phi_{V} (\tau_q(u'_1,v'_1,...,u'_{q-1},v'_{q-1}), V_q), \]
so that the joint law of $(B_{t_q},u'_q,v'_q)_{q=1,...,N-1}$ defined on $(\bar\Omega,\bar\FR,\bar\P)$ is the same as the law of $(B_{t_q},u_q,v_q)_{q=1,...,N-1}$ under $\P_{t,m,\pi,\sigma^*,\tau}$. 
\pa
Thanks to property $b)$ the conditional law of $B_{t_q}$ given $(M'_{t_i},u'_i,v'_i)_{i=1,...,q}$ equals the conditional law of $B_{t_q}$ given $(M'_{t_1},...,M'_{t_q})$ and thus is exactly  $M'_{t_q}$. To prove this, we use first that the actions of player 2 are maps depending on $(u'_i)_{i=1,...,q-1}$ and auxiliary variables $(V_i)_{i=1,..,q}$ that are independent of $(B_{t_q}, (M'_{t_i},u'_i)_{i=1,..q})$, implying that the variables $(v'_i)_{i=1,...,q}$ can be removed from the conditioning. Then, the variables $(u'_i)_{i=1,...,q}$ can be removed as well since they are maps depending on  $(M'_{t_i})_{i=1,...,q}$. 
\pa
Using the above-mentioned properties and inequality \ref{ineqselec}, we obtain:
\begin{align*}
\gamma_\pi(t,m,\sigma^*,\tau)=&\bar\E\left[ \sum_{q=1}^{N-1}(t_{q+1}-t_q)f(t_q,B_{t_q},u'_q,v'_q)\right]\\
&= \bar\E\left[ \sum_{q=1}^{N-1}(t_{q+1}-t_q) \bar\E[f(t_q,B_{t_q},u'_q,v'_q) |(M'_{t_i},u'_i,v'_i)_{i=1,...,q} ] \right] \\
&= \bar\E\left[ \sum_{q=1}^{N-1}(t_{q+1}-t_q)  \int_{\R^d } f(t_q,x,u'_q,v'_q) dM'_{t_q}(x) \right] \\
&\leq 
\bar\E\left[ \sum_{q=1}^{N-1} (t_{q+1}-t_q)H(t_q,M'_{t_q})  \right]=\E\left[ \sum_{q=1}^{N-1}  (t_{q+1}-t_q)  H(t_q,M_{t_q})\right].
\end{align*}
Since the strategy $\tau$ was chosen arbitrarily, it follows that
\[ V_\pi(t,m)\leq \bar\E \left[ \sum_{q=1}^{N-1} (t_{q+1}-t_q)H(t_q,M_{t_q})  \right].\]
Letting tend $|\pi|$ to zero on both sides of the latter inequality and using Lemma \ref{Hriemann}, the relation \eqref{ineqlimsup} follows
\end{proof}

\begin{Proposition}\label{liminf} 
\begin{equation}\label{ineqliminf}
\liminf_{|\pi| \rightarrow 0} V_{\pi}(t,m)\geq V(t,m).
\end{equation}
\end{Proposition}

\begin{proof}
For some given partition $\pi$ of $[t,T]$, let $\sigma\in\Sigma(\pi)$ be an arbitrary strategy for player 1. We shall define recursively an answer $\bar\tau\in\TR$ from player 2. Since it will be a pure strategy for each $q\in\{ 1, \ldots,N-1\}$, $\bar\tau_q$ will be identified with a map from $(U\times V)^{q-1}$ to  $V$.
\begin{itemize}
\item let $v^* :(s,m')\in[t,T]\times\PR_2\mapsto v^*(t,m')\in U$ be a measurable selection of
\[ \mbox{Argmax}_{v\in V}\min_{u\in U}\int_{\R^d}f(t_q,x,u,v)dm'(x).\]
\item For $q=1$, set $\bar\tau_1=v^*(t,m)$.
\item Suppose that, for some $q\in\{ 2,\ldots,N\}$, $\bar\tau_1,\ldots,\bar\tau_{q-1}$ is defined. Remark that, for any couple $(\sigma,\tau)\in\Sigma(\pi)\times\TR(\pi)$, the restriction of $\P_{t,m,\pi,\sigma,\tau}$ on the coordinates $(B_{t_1},u_1,v_1,\ldots,B_{t_q},u_q)$ depends on $\tau$ only through $\tau_1,\ldots,\tau_{q-1}$.
Therefore it makes sense to define $\bar\tau_q$ as
\[ \bar\tau_q=v^*(t_q,\hat M^{\pi}_{t_q}),\]
where $\hat M^{\pi}_{t_q}$ is the conditional law of $B_{t_q}$ given $(u_1,v_1,\ldots,u_{q-1},v_{q-1})$ under $\P_{t,m,\pi,\sigma,\bar\tau}$ for $q=2,...,N-1$, and $\hat{M}^{\pi}_{t_1}=m$. 
\end{itemize}
We also need to define for all $q=1,...,N-1$, the variable $M^{\pi}_{t_q}$ as the conditional law of  of $B_{t_q}$ given $(u_1,v_1,\ldots,u_{q},v_{q})$ under $\P_{t,m,\pi,\sigma,\bar\tau}$. Note that the variables $\hat M^\pi_{t_q}$ and $M^\pi_{t_q}$ correspond respectively to the belief of player $2$ on $B_{t_q}$ before and after playing round $q$, when knowing that the strategy $\sigma$ is used by player $1$. Note also that by construction, for $q=2,..,N-1$,
\[ \hat{M}^{\pi}_{t_q}=p^{t_{q-1},M^{\pi}_{t_{q-1}}}_{t_{q}}.\]
At last, we extend the definition of the process $(M^{\pi}_s)_{s \in [t,T]}$ as follows:
\[ \forall q=1,...,N-1, \forall s\in [t_q,t_{q+1}), \; M^{\pi}_s=p^{t_q,M^{\pi}_{t_q}}_s, \; M^\pi_T=p^{t_{N-1},M^\pi_{t_{N-1}}}_T.\]
By construction, the law of $M^\pi$ belongs to $\tilde\MR(t,m)$. 
\pa
Let us start the computations. For all $q=1,...,N-1$
\begin{align*}
\E_{t,m,\pi,\sigma,\bar\tau}[f(t_q,B_{t_q},u_q,v_q)]&=
\E_{t,m,\pi,\sigma,\bar\tau}\left[\E_{t,m,\pi,\sigma,\bar\tau}[f(t_q,B_{t_q},u_q,v_q)|u_1,v_1,\ldots,u_{q},v_q]\right]\\
&=
\E_{t,m,\pi,\sigma,\bar\tau}\left[\int_{\R^d}f(t_q,x,u_q,v_q)d M^\pi_{t_q}(x)\right]\\
&\geq\E_{t,m,\pi,\sigma,\bar\tau}\left[\int_{\R^d}f(t_q,x,u_q,v_q)d \hat{M}^\pi_{t_q}(x) -Cd_1(M^\pi_{t_q},\hat{M}^\pi_{t_q})\right]
\\
& \geq \E_{t,m,\pi,\sigma,\bar\tau}\left[ H(t_q,\hat{M}^\pi_{t_q})  -Cd_1(M^\pi_{t_q},\hat{M}^\pi_{t_q})\right]\\
& \geq \E_{t,m,\pi,\sigma,\bar\tau}\left[ H(t_q,M^\pi_{t_q})  -2Cd_1(M^\pi_{t_q},\hat{M}^\pi_{t_q})\right].
\end{align*}
Taking the sum over all $q$, we have proven that, for all $\sigma\in\Sigma(\pi)$, there exists $\bar\tau\in\TR(\pi)$ and $M^\pi\in \MR(t,m)$ such that
\begin{align*}
\gamma_{\pi}(t,m,\sigma,\bar\tau)\geq &
\E_{t,m,\pi,\sigma,\bar\tau}\left[\sum_q (t_{q+1}-t_q) \left(H(t_q, M^\pi_{t_q})-2Cd_1(M^\pi_{t_q},\hat{M}^\pi_{t_q})\right)\right].\\
\end{align*} 
Let $(\pi_n)_{n \geq 1}$ denote a sequence of partitions such that $|\pi_n| \rightarrow 0$ and 
\[ \lim_n V_{\pi_n}(t,m)= \liminf_{|\pi|\rightarrow 0} V_\pi (t,m).\]
Let $(\varepsilon_n)_{n \geq 1}$ a sequence of positive numbers with limit $0$ and for all $n\geq 1$, let $\sigma_n \in \Sigma(\pi_n)$ be an $\varepsilon_n$-optimal strategy in the game $\Gamma_{\pi_n}(t,m)$. Thanks to the above analysis, there exists $\tau_n \in \TR(\pi_n)$ such that
\begin{align*}
V_{\pi_n}(t,m) +\varepsilon_n &\geq  \gamma_{\pi_n}(t,m,\sigma_n,\tau_n) \\
&\geq 
\E_{\pi_n,t,m,\sigma_n,\tau_n}\left[\sum_q (t_{q+1}-t_q) \left(H(t_q, M^{\pi_n}_{t_q})-2Cd_1(M^{\pi_n}_{t_q},\hat{M}^{\pi_n}_{t_q})\right)\right].
\end{align*}
 
\pa By Proposition \ref{Mcompact}, there exists a subsequence of $(M^{\pi_n})_{n \geq 1}$ which converges for $\LR(\kappa_1)$ to some process $M$ with law in $ \MR(t,m)$. 
After some eventual enlargement of $(\Omega,\FR,\P)$ we may consider representations of  $M$  and the sequence $(M^{\pi_n})$  on this same space. 
Using Lemmas \ref{d1var} and \ref{Hriemann}, we get finally
\[ \liminf_{|\pi|\rightarrow 0}V_\pi(t,m)\geq \E[\int_t^T H(s,M_s)ds]\geq V(t,m).\]

\end{proof}

\section{Characterization of the value function}
\label{edp}

\pa The main result of this paper is the characterization of the value function $V$ as the largest subsolution, in some class of functions, of a Hamilton-Jacobi equation. In contrast with previous works, $V$ depends here in a dynamic way on the probability measures $m\in\PR_2$ which are no more of finite support, neither can be forced to have a density with respect to the Lebesgue-measure.  We have found the appropriate framework in the paper \cite{CDLL}.\pa
For technical reasons (see Remark \ref{remarktestfunctions} below), we need to consider the two metrics $d_1$ and $d_2$ on the space $\PR_2$. The reference metric is $d_1$ and is used implicitly everywhere, and we will say explicitly $d_2$-continuous, $d_2$-convergent, etc... whenever we need to use the metric $d_2$. In particular, functions defined on $\PR_2$ or $[0,T]\times \PR_2$, $[0,T]\times \PR_2\times \R^d$ are $d_2$-continuous if they are globally continuous when $\PR_2$ is endowed with the metric $d_2$ and the other spaces with the usual topology.
\pa 
Given a map $F:[0,T]\times\PR_2 \rightarrow \R$, we consider the following equation:
\begin{equation}
\label{eq1}
\left\{
\begin{array}{l}
\partial_tU(t,m)+\frac{1}{2}\int_{\R^d}\di[D_mU](t,m,x)m(dx)+F(t,m)=0,\; (t,m)\in[0,T)\times\PR_2,\\
\\
U(T,m)=0,\; m\in\PR_2,
\end{array}
\right.
\end{equation}
where the divergence operator $\di$ acts on the spatial variable $x$ and $D_mU (t,m,x)$ is defined in \cite{CDLL} (Definition 2.2.1) as follows:
\pa
At first, a map $g :\PR_2 \mapsto \R$ is said to be differentiable if there exists a measurable map $\frac{\delta g}{\delta m}:\PR_2\times\R^d\rightarrow\R$ with at most quadratic growth with respect to the spatial variable and such that, for all $m,m'\in\PR_2$,
\[ \lim_{r\rightarrow 0^+}\frac{g\left( (1-r)m+rm'\right)-g(m)}{r}=
\int_{\R^d}\frac{\delta g}{\delta m}(m,x)d(m'-m)(x),\]
and (as a normalization convention)
\[ \int_{\R^d}\frac{\delta g}{\delta m}(m,x)dm(x)=0.\]
\pa
As a consequence of the definition, if $g$ is differentiable, the map
\[ r \in [0,1] \rightarrow g((1-r)m+rm') \in \R,\]
is differentiable on $(0,1)$ with derivative $\int_{\R^d}\frac{\delta g}{\delta m}((1-r)m+rm',x)d(m'-m)(x)$. If moreover, $\frac{\delta g}{\delta m}$ is $d_2$-continuous 
and with at most quadratic growth with respect to $x$, uniformly in $m$, the above derivative is $C^1$ (using dominated convergence) and we may apply the fundamental theorem of calculus, yielding the relation
\be \label{eqprimitive} 
g(m')-g(m)= \int_0^1  \int_{\R^d}\frac{\delta g}{\delta m}((1-r)m+rm',x)d(m'-m)(x)dr. 
\ee
\pa
Then,  if $\frac{\delta g}{\delta m}(m,x)$ is $C^1$ in $x$, we define the intrinsic derivative $D_m g$ by
\[ D_m g(m,x):=D_x\left(\frac{\delta g}{\delta m}\right)(m,x).\]
We refer to \cite{CDLL} for the interpretation of $D_m g$, which is related to the geometry of the Wasserstein space, coincides with the $L^2$-derivative considered by Lions \cite{Lions} and appears quite naturally as a measure derivative along vector fields (see Proposition 2.3 in \cite{CDLL}).
\pa 
In the sequel, we are mainly interested in the case where $F=H$ :
\begin{equation}
\label{eqH}
\left\{
\begin{array}{l}
\partial_tU(t,m)+\frac{1}{2}\int_{\R^d}\di[D_mU](t,m,x)m(dx)+H(t,m)=0,\; (t,m)\in[0,T)\times\PR_2,\\
\\
U(T,m)=0,\; m\in\PR_2.
\end{array}
\right.
\end{equation}
\pa 

\pa   
The problem is that $H$ isn't sufficiently regular. Indeed a crucial tool in our argumentation is a classical, explicit solution for some equations of type \eqref{eq1} (see Lemma \ref{techlemma}), which only exists under strong regularity assumptions. For this reason we need to approach $H$ by smooth functions $F$. 
\pa
The regularity assumptions $(A1)$ we require for the function $F$ in (\ref{eq1}) are the following:
 \[ (A1) 
\left\{
\begin{array}{l}
\mbox{The map } (t,m)\in[0,T]\times\PR_2\mapsto F(t,m) \mbox{ is bounded and continuous}, \\
 \mbox{$\frac{\delta F}{\delta m}, D_x \left(\frac{\delta F}{\delta m}\right), D^2_x\left(\frac{\delta F}{\delta m}\right)$ exist
 and are bounded and continuous}.
\end{array}\right.\]

\pa We call $(A2)$ the regularity assumptions for functions $\psi$ that will play the role of a terminal condition for equation (\ref{eq1}):
\[ (A2) 
\left\{
\begin{array}{l}
\mbox{The map } m\in\PR_2\mapsto \psi(m)\mbox{ is continuous with at most linear growth, i.e.:}\\
 \mbox{there exists $M>0$ such that $\forall m \in \PR_2$, $|\psi(m)|\leq M(1+|m|_1)$},\\
 \mbox{$\frac{\delta \psi}{\delta m}, D_x \left(\frac{\delta \psi}{\delta m} \right), D^2_x\left( \frac{\delta \psi}{\delta m}\right)$ exist, are continuous,
}\\
\mbox{and have linear growth in $x$, uniformly in $m$}.
\end{array}\right.\]

\pa We introduce also the weaker assumption $(A3)$ for test functions $\varphi$ in order to define the notion of viscosity subsolution for equation \ref{eq1}:
\[  (A3)
\left\{
\begin{array}{l}
\mbox{ The map }(t,m)\in[0,T]\times\PR_2\mapsto \varphi(t,m) \mbox{ is lower semi-continuous and $d_2$-continuous,}\\
 \mbox{ $\frac{\partial \varphi}{\partial t},\frac{\delta \varphi}{\delta m}, D_x\left(\frac{\delta \varphi}{\delta m}\right),D^2_x \left(\frac{\delta \varphi}{\delta m} \right)$ exist and are $d_2$-continuous on $[0,T)\times \PR_2\times \R^d$,} \\
\mbox{ $\frac{\delta \varphi}{\delta m},D_x\left(\frac{\delta \varphi}{\delta m}\right),D^2_x \left(\frac{\delta \varphi}{\delta m} \right)$ have at most quadratic growth with respect to $x$, uniformly in $(t,m)$}
\end{array}\right.\]
\pa
\begin{Remark}\label{remarktestfunctions}{\rm
We need to consider $\tilde\varphi(t,m)=\varphi(t,m)+\epsilon|m|_2^2$ as an admissible test function in Lemma \ref{clas}, where the second term is related to a compactness issue. Actually, $\varphi$ will be more regular than required by $(A3)$, but note that the map $\psi(m)=|m|_2^2$ is only lower semi-continuous and not continuous, but is $d_2$-continuous.
 An easy computation shows that $\frac{\delta \psi}{\delta m}(m,x)=|x|^2-|m|_2^2$, and thus $\psi$ satisfies $(A3)$. The above assumptions are thus completely tailored to the problem, and we would like to emphasize that there are several different ways to define test functions or classical solutions in order to obtain a coherent notion of (viscosity) solutions on measure spaces. }
\end{Remark}
\pa

\begin{Definition}
\begin{enumerate}
\item We call a (classical) {\em solution} of \eqref{eq1} a map $U:[0,T]\times\PR_2\rightarrow\R$ which is continuous, satisfies assumption (A3) and for which \eqref{eq1} is satisfied for all $(t,m)\in[0,T]\times\PR_2$.
\item We call a {\em subsolution} of (\ref{eq1}) a map $U:[0,T]\times\PR_2\rightarrow\R$ which is upper semi-continuous and satisfies, for all $(t,m)\in[0,T)\times\PR_2$, and for all $\varphi$ satisfying assumption (A3), such that $\varphi-U$ has a local minimum at $(t,m)$,
 \begin{equation}
 \label{testf}
  \partial_t\varphi(t,m)+\frac{1}{2}\int_{\R^d}\di[D_m\varphi](t,m,x)m(dx)+F(t,m)\geq 0.
  \end{equation}
\end{enumerate}
 \end{Definition}
 
\begin{Remark}{\rm It is easy to prove that a classical solution of  \eqref{eq1} is also a subsolution, but there is no need here for this result. In the definition of subsolution, local refers to the metric $d_1$ on $\PR_2$.
}
\end{Remark}

We will prove in this section the following characterization of the value function $V$:

 \begin{Theorem}
 \label{bigsub}
 $V$ is the largest bounded and continuous subsolution of (\ref{eqH}) which is convex in $m$ and satisfies the terminal condition $V(T,\cdot)=0$.
 \end{Theorem}
 
\pa We start our computations with a technical lemma in which we construct smooth solutions of \eqref{eq3} with smooth terminal conditions and compute the derivative of test functions along the curve $s \mapsto (s,p^{t,m}_s)$.

\begin{Lemma}\label{techlemma}
1) If $\varphi$ is a test function satisfying the assumptions $(A3)$, then for all $(t,m)\in [0,T)\times \PR_2$: 
\[ \lim_{s \searrow t}\frac{\varphi(s,p^{t,m}_s)-\varphi(t,m)}{s-t}=\partial_t\varphi(t,m)+\frac{1}{2}\int_{\R^d}\di[D_m\varphi](t,m,x)m(dx) .\] 
2) Let $F:[0,T]\times\PR_2\rightarrow\R$ be such that assumptions $(A1)$ hold true and let $\psi:\PR_2\rightarrow\R$ which satisfies $(A2)$. Let $t_1\in[0,T]$. 
Then the following equation
\begin{equation}
\label{eq3}
\left\{ 
\begin{array}{l}
\partial_t\varphi(t,m)+\frac{1}{2}\int_{\R^d}\di [D_m\varphi (t,m,x)]dm(x)+F(t,m)=0, \; (t,m)\in[0,t_1)\times\PR_2,\\
\varphi(t_1,m)=\psi(m), \; m\in\PR_2,
 \end{array}
 \right.
 \end{equation}
 admits as unique classical solution the continuous map $\varphi$ which satisfies $(A3)$ and is defined by: 
 \begin{equation}
\label{eqphi} 
 \varphi(t,m)=\psi(p^{t,m}_{t_1})+\int_t^{t_1}F(s,p^{t,m}_s)ds. 
 \end{equation}

\end{Lemma}

\begin{Remark}{\rm A direct application of this lemma is that, if $H$ satisfies assumption (A1), then the map $(t,m)\mapsto U_0(t,m):=\int_t^TH(s,p^{t,m}_s)ds$ is a  classical solution of \eqref{eqH}. But remark that we cannot expect this map to be convex in $m$, unless this holds for $H$ itself. This implies that, in general, $U_0\neq V$ (i.e. the totally non revealing strategy is not always optimal for player 1).
} \end{Remark}

\begin{proof}[Proof of Lemma \ref{techlemma}] $~~$\pa
1.1) Fix $0\leq t<t'<s\leq T$ and  let us assume temporarily that $m$ has compact support. Set $m_s=p^{t,m}_s$ and $m_{t'}=p^{t,m}_{t'}$.
\pa
Recall that 
\[\frac{\partial}{\partial s} (\rho^{t,m}_s(x))=\frac{1}{2}\Delta \rho^{t,m}_s(x)\]
We have the following chain of equalities:
\begin{align*}
\varphi(s,m_s)-\varphi(s,m_{t'})&=
 \int_0^1\int_{\R^d}\frac{\delta \varphi}{\delta m}\left(s, (1-r)m_{t'}+rm_s,x\right)d(m_s-m_{t'})(x)dr\\
&= \int_0^1\int_{\R^d}\frac{\delta \varphi}{\delta m}\left(s, (1-r)m_{t'}+rm_s,x\right)(\rho^{t,m}_s-\rho^{t,m}_{t'})(x)dx dr \\
&= \int_0^1\int_{\R^d}\frac{\delta \varphi}{\delta m}\left(s, (1-r)m_{t'}+rm_s,x\right)\left(\int_{t'}^s\partial_\tau\rho^{t,m}_\tau(x)d\tau \right) dx dr\\
&= \int_{t'}^s\int_0^1\int_{\R^d}\frac{\delta \varphi}{\delta m}\left(s,(1-r)m_{t'}+rm_s,x\right)\partial_\tau\rho^{t,m}_\tau(x) dxdrd\tau\\
&=\int_{t'}^s\int_0^1\int_{\R^d}\frac{\delta \varphi}{\delta m}\left( s,(1-r)m_{t'}+m_s,x\right)\frac{1}{2}\Delta \rho^{t,m}_\tau(x) dxdrd\tau\\
&=\frac{1}{2}\int_{t'}^s\int_0^1\int_{\R^d}\di[D_m\varphi]\left( s,(1-r)m_{t'}+rm_s,x\right) dp^{t,m}_\tau(x) drd\tau.
\end{align*}
Let us justify the above computations. The first equality follows from \eqref{eqprimitive}. The second equality follows from the definition of $m_s,m_{t'}$. The fourth equality follows from Fubini's theorem, which we may apply since $\frac{\delta \varphi}{\delta m}$ has at most quadratic growth in $x$, uniformly in $(t,m)$, and -thanks to the fact that $m$ has compact support- $\int_{\R^d} (1+|x|^2) |\frac{\partial}{\partial s} (\rho^{t,m}_s(x))| dx <\infty$. The sixth and last equality follows from the integration by part formula which we may apply thanks to the growth assumptions on $\frac{\delta \varphi}{\delta m}$ and its derivatives.
\pa 
We proved that, for $m$ with compact support,
\be \label{eqderiv1} 
\varphi(s,m_s)-\varphi(s,m_{t'})=\frac{1}{2}\int_{t'}^s\int_0^1\int_{\R^d}\di[D_m\varphi]\left(s,(1-r)m_{t'}+rm_s,x\right) dp^{t,m}_\tau(x) drd\tau .
\ee
1.2) 
Let us generalize this for an arbitrary $m\in\PR_2$: any $m \in \PR_2$ is the $d_2$-limit of a sequence $(m^n)$ of measures with finite (hence compact) support. 
Following Lemma \ref{dp}, this implies that, for any $\tau\in[t',s]$, $p^{t',m^n}_\tau$ $d_2$-converges to $p^{t',m}_\tau$ and therefore, denoting $m^n_s=p^{t,m^n}_s$:
\begin{equation}
\label{varphit}
 \varphi(s,m^n_s)-\varphi(s,m^n_{t'})\rightarrow \varphi(s,m_s)-\varphi(s,m_{t'}).
 \end{equation}
Further, the $d_2$-convergence of $(m^n)$  implies that their second order moments are bounded and uniformly integrable and that the sequence of measures is tight, i.e.
\[  \sup_n \int_{\R^d} |x|^2 dm^n(x)<\infty, \; \sup_n \int_{|x|> K} (1+|x|^2) dm^n(x) \underset{K \rightarrow \infty}{\longrightarrow} 0.\]
Moreover, we also have 
\[ \sup_{\tau \in (t,s]} \int_{|x|> K} (1+|x|^2) dp^{t,\delta_0}_{\tau}(x)= \int_{|x|> K} (1+|x|^2) dp^{t,\delta_0}_{s}(x) \underset{K \rightarrow \infty}{\longrightarrow} 0.\]
Thus, for any $\varepsilon>0$, we can find $K$ sufficiently large such that, for all $\tau \in (t,s]$ and all $n\in\N$,
\begin{align*}
\int_{|x|> K} (1+|x|^2) dp^{t,m^n}_{\tau}(x)&=\int_{\R^d}\int_{\R^d}\ind_{|x+y|>K}  (1+|x+y|^2)\rho_{\tau-t}(y)dy dm^n(x) \\ 
&\leq \int_{\R^d}\int_{\R^d} (\ind_{|y|> K/2}+ \ind_{|x|>K/2})(1+2|x|^2+2|y|^2)\rho_{\tau-t}(y)dy dm^n(x),\\
&\leq \varepsilon,
\end{align*}
as well as
$\int_{|x|> K} (1+|x|^2) dp^{t,m}_{\tau}(x)\leq \varepsilon$.\\
For this $K>0$, we have, for some constant $C'$,
\begin{equation}
\begin{array}{rl}
\label{mn}
|\frac{1}{2}\int_{t'}^s\int_0^1\int_{\R^d}\di[D_m\varphi]&\left( s,(1-r)m^n_{t'}+rp^{t,m^n}_s,x\right) d(p^{t,m}_\tau-p^{t,m^n}_{\tau})(x) dr d\tau| \\ \leq &C'\Big(
\int_{t'}^s\int_0^1\int_{|x|> K} (1+|x|^2) d(p^{t,m}_{\tau}(x)+p^{t,m^n}_{\tau})(x) dr d\tau\\
&\qquad\qquad+
\int_{t'}^s\int_0^1
\int_{|x|\leq K}(1+|x|^2) |\rho^{t,m}_\tau(x)-\rho^{t,m^n}_{\tau}(x)|dxdr d\tau\Big)\\
\leq& 2C'T\varepsilon+C'\int_{t'}^s\int_0^1
\int_{|x|\leq K}(1+|x|^2) |\rho^{t,m}_\tau(x)-\rho^{t,m^n}_{\tau}(x)|dxdr d\tau.
\end{array}
\end{equation}
By dominated convergence, the last term of \eqref{mn} converge to zero as $n$ goes to infinity, and we conclude that
\[ 
|\frac{1}{2}\int_{t'}^s\int_0^1\int_{\R^d}\di[D_m\varphi]\left( s,(1-r)m^n_{t'}+rp^{t,m^n}_s,x\right) d(p^{t,m}_\tau-p^{t,m^n}_{\tau})(x) dr d\tau|\underset{n\rightarrow}{\longrightarrow} 0.\]

\pa Using dominated convergence, we also have 
\begin{multline*}
\frac{1}{2}\int_{t'}^s\int_0^1\int_{\R^d}\di[D_m\varphi]\left( s,(1-r)m^n_{t'}+rp^{t,m^n}_s,x\right) dp^{t,m}_{\tau}(x) dr d\tau \\ \underset{n \rightarrow \infty}{\longrightarrow} \frac{1}{2}\int_{t'}^s\int_0^1\int_{\R^d}\di[D_m\varphi]\left( s,(1-r)m_{t'}+rp^{t,m}_s,x\right) dp^{t,m}_{\tau}(x)dr d\tau.
\end{multline*}

 \pa These two limits together with \eqref{varphit} imply that \eqref{eqderiv1} holds true also for any $m \in \PR_2$.

\pa
1.3) The next step is to prove that \eqref{eqderiv1} still holds true for $t'=t$, i.e.: 
\be \label{eqderiv2} 
\varphi(s,m_s)-\varphi(s,m)=\frac{1}{2}\int_{t}^s\int_0^1\int_{\R^d}\di[D_m\varphi]\left(s,(1-r)m+rm_s,x\right) dp^{t,m}_\tau(x) drd\tau .
\ee
By Lemma \ref{dp},  $\tau\mapsto p^{t,m}_\tau$ is continuous from $[t,T]$ to $(\PR_2,d_2)$. And by (A3), $\varphi$ and $\di[D_m\varphi]$ are $d_2$-continuous. Therefore, letting tend $t'$ to $t$ in \eqref{eqderiv1} , we obtain \eqref{eqderiv2} (the right hand side converges by dominated convergence).\\

\pa  1.4) Finally we shall divide \eqref{eqderiv2} by $s-t$ and let tend $s$ to $t$.\\
Since $\di[D_m\varphi](t,m,.)$ has at most quadratic growth and $\tau \mapsto p^{t,m}_{\tau}$ is continuous from $[t,s]$ to $(\PR_2,d_2)$, 
\[ \tau \mapsto \int_{\R^d}\di[D_m\varphi](t,m,x)dp^{t,m}_{\tau}\]
is continuous. This implies that
\[\lim_{s\searrow t} \frac{1}{s-t}\int_{t}^s \int_{\R^d}\di[D_m\varphi](t,m,x)dp^{t,m}_{\tau}(x)d\tau = \int_{\R^d}\di[D_m \varphi](t,m,x)dm(x).\]
Fix $\varepsilon>0$. As for \eqref{mn},using that $\di[D_m\varphi]$ has at most quadratic growth in $
x$ uniformly in $(t,m)$, we can find $K>0$ such that, for all $\tau \in (t,s]$,
\[ \int_{|x|>K}\Big|\di[D_m\varphi]\left( s,(1-r)m+rm_s,x\right)-\di[D_m\varphi]\left( t,m,x\right)\Big|dp^{t,m}_{\tau}(x)\leq\varepsilon/2.\]
Further, by assumption (A3), $\di[D_m \varphi]$ is $d_2$-continuous. Therefore
we can choose $s-t$ sufficiently small so that for any $r\in[0,1]$ and $x\in\R^d$, with $|x|\leq K$,
\[
 |\di[D_m\varphi]\left( s,(1-r)m+rm_s,x\right)-\di[D_m\varphi]\left( t,m,x\right)| \leq \varepsilon/2.
\]
in order to get
\[ \int_0^1\int_{\R^d}|\di[D_m\varphi]\left( s,(1-r)m+rm_s,x\right)-\di[D_m\varphi]\left( t,m,x\right)|dp^{t,m}_{\tau}(x)d\tau\leq\varepsilon.\]
Since $\varepsilon$ can be taken arbitrarily small, it derives from \eqref{eqderiv2} that
\[
\lim_{s\searrow t}\frac{\varphi(s,m_s)-\varphi(s,m)}{s-t}=
\frac{1}{2}\int_{\R^d}\di[D_m \varphi](t,m,x)dm(x).\]
Writing $\varphi(s,m_s)-\varphi(t,m)=\varphi(s,m_s)-\varphi(s,m) +\varphi(s,m)-\varphi(t,m)$ and recalling that
\[
\lim_{s\searrow t}\frac{\varphi(s,m)-\varphi(t,m)}{s-t}=\partial_t \varphi(t,m),\]
we can conclude.
\pa
2.1) The function $\varphi$ defined in \eqref{eqphi} is solution of \eqref{eq3}:\\
Remark that, for all $m,m'\in\PR_2$ and $r\in[0,1]$ and $s\in [t,T]$, we have the relation 
\[ p^{t,(1-s)m+sm'}_r=(1-s)p^{t,m}_r+sp^{t,m'}_r.\]
We have therefore:
\begin{align*} 
\frac{1}{r}\left(\varphi(t,(1-r)m+rm')-\varphi(t,m)\right)
&= \frac{1}{r}\left(\psi((1-r)p^{t,m}_{t_1}+rp^{t,m'}_{t_1})-\psi(p^{t,m}_{t_1})\right) \\ &\qquad +\int_t^{t_1} \frac{1}{r}\left(F(s,(1-r)p^{t,m}_{s}+rp^{t,m'}_{s})-F(s,p^{t,m}_s)\right)ds
\end{align*}
Using that $\psi$ satisfies $(A2)$ and that $F$ satisfies $(A1)$, together with bounded convergence and Fubini's theorem, we deduce that $\frac{\delta\varphi}{\delta m}$ exists and is given by:
\[  \frac{\delta\varphi}{\delta m}(t,m,x)=\int_{\R^d}\rho_{t_1-t}(y)\left(\frac{\delta\psi}{\delta m}(p^{t,m}_{t_1},x+y) +\int_t^T \frac{\delta F}{\delta m} (s,p^{t,m}_s,x+y)ds\right)dy.\]
We deduce easily from the assumptions on $\psi,F$ that $\varphi$ satisfies $(A3)$.
\pa
It follows from point $1)$ that:
\[
\lim_{s\searrow t}\frac{\varphi(s,m_s)-\varphi(t,m)}{s-t}=\partial_t\varphi(t,m)+
\frac{1}{2}\int_{\R^d}\di[D_m \varphi](t,m,x)dm(x).\]
Further, from the semigroup property (\ref{Markov}), we get
\[
 \varphi(s,m_s)=\psi(p_{t_1}^{s,m_s}) + \int_s^T F(r,p^{s,m_s}_r)dr=\psi(p^{t,m}_{t_1}) + \int_s^T F(r,p^{t,m}_r) dr,\; s\in[ t,T].
 \]
 This implies that
\[
\varphi(s,m_s)-\varphi(t,m)=-\int_t^s F(r,p^{t,m}_r)dr,
\]
and thus:
\[ \lim_{s\searrow t}\frac 1{s-t}\left( \varphi(s,m_s)-\varphi(t,m)\right)=-F(t,m),\]
which implies that $\varphi$ is a solution of (\ref{eq3}).\\

\pa 2.2) The function $\varphi$ is the unique solution of \eqref{eq3}:\\
Let $\varphi'$ be an arbitrary solution of (\ref{eq3}). It follows from point $1)$ together with the semigroup property (\ref{Markov}) that the map
\[ s \in [t,T] \mapsto \varphi'(s,p^{t,m}_s) \in \R,\]
is differentiable on $(t,T)$ with derivative $\partial_t \varphi'(s,p^{t,m}_s)+\frac{1}{2}\int_{\R^d}\di [D_m \varphi'](s,p^{t,m}_s,y)\rho^{t,m}_s(y) d(y)$.
Using dominated convergence, it is even $C^1$ on $(t,T)$, and we may apply the fundamental theorem of calculus to obtain 
 \[\varphi'(t_1,p^{t,m}_{t_1}) -\varphi'(t,m)=\int_t^{t_1}\left(\partial_t \varphi'(s,p^{t,m}_s)+\frac{1}{2}\int_{\R^d}\di [D_m \varphi'](s,p^{t,m}_s,y)dp^{t,m}_s(y)\right)ds\]
Then
\[\begin{array}{rl}
\varphi'(t,m)=& \varphi'(t,m)-\varphi'(t_1,p^{t,m}_{t_1})+ \psi(t_1,p^{t,m}_{t_1})\\
=&-\int_t^{t_1}\left(\partial_t \varphi'(s,p^{t,m}_s)+\frac{1}{2}\int_{\R^d}\di [D_m \varphi'](s,p^{t,m}_s,y)dp^{t,m}_s(y)\right)ds+ \psi(t_1,p^{t,m}_{t_1})\\
=&\int_t^{t_1} F(s,p^{t,m}_s)ds+ \psi(t_1,p^{t,m}_{t_1}),
\end{array}\]
i.e. $\varphi'(t,m)=\varphi(t,m)$.
\end{proof}


 \pa Here is the a first application of Lemma \ref{techlemma}:
 
 \begin{Proposition}
The value function  $V$ is a subsolution of (\ref{eqH}).
\end{Proposition}

\begin{proof}
Fix $t\leq s\leq T$ and $m\in\PR$.\\
For $\epsilon>0$, let $\bar M=(\bar M_r)_{r\in[s,T]}\in\MR(s,p^{t,m}_s)$ be $\epsilon$-optimal for $V(s,p^{t,m}_s)$.
We define another process $M'$ on the time interval $[t,T]$ by
\[ M'_r=\left\{\begin{array}{ll}
p^{t,m}_r,& \mbox{ if } r\in[t,s),\\
\bar M_r,& \mbox{ if }r\in[s,T].
\end{array}
\right.\]
Let us check that $(M'_r)_{r\in[t,T]}$ belongs to $\MR(t,m)$:
\begin{itemize}
\item It is clear that $\E[M'_t]=m$.
\item To prove that $M'$ satisfies item {\em (ii)} of the definition, let $\phi:\R^d\rightarrow\R$ be continuous and bounded. 
\begin{itemize}
\item  $r\mapsto  M'_r(\phi)$ is c\`adl\`ag,
\item for $s\leq t_1\leq t_2\leq T$,
\[\begin{array}{rl}
\E[M'_{t_2}(\phi)|\FR^{M'}_{t_1}]=&
\E[\bar M_{t_2}(\phi)|\FR^{\bar M}_{t_1}]\\
=&p^{t_1,\bar M_{t_1}}_{t_2}(\phi)=p^{t_1,M'_{t_1}}_{t_2}(\phi).
\end{array}\]
\item If $t\leq t_1< s\leq t_2\leq T$, 
since $\FR^{M'}_{t_1}=\FR^{p^{t,m}}_{t_1}$ is trivial, we have
\begin{align*}
\E[M'_{t_2}(\phi)|\FR^{M'}_{t_1}]=&\E[\bar M_{t_2}(\phi)]\\
=&p^{s,p^{t,m}_s}_{t_2}(\phi) \mbox{ (by \eqref{propM3})}\\
=&p^{t,m}_{t_2}(\phi)=p^{t_1,p^{t,m}_{t_1}}_{t_2}(\phi) \mbox{ (by the semigroup property \eqref{Markov})}\\
=&p^{t_1,M'_{t_1}}_{t_2}(\phi).
\end{align*}

 \item If $t\leq t_1\leq t_2< s$,  we have simply
 \[\begin{array}{rl}
\E[M'_{t_2}(\phi)|\FR^{M'}_{t_1}]=&
p^{t_1,p^{t,m}_{t_1}}_{t_2}(\phi)=p^{t_1,M'_{t_1}}_{t_2}(\phi).
\end{array}\]
 \end{itemize}
 \end{itemize}

 \pa We can write now
\[\begin{array}{rl}
V(t,m)=&\inf_{M\in\MR(t,m)}\E\left[\int_t^TH(r,M_r)dr\right]\\
\leq &\E\left[\int_t^TH(r,M'_r)dr\right]=\int_t^sH(r,p^{t,m}_r)dr+\E\left[\int_s^TH(r,\bar M_r)dr\right]\\
\leq & \int_t^sH(r,p^{t,m}_r)dr+V(s,p^{t,m}_s)+\epsilon.
\end{array}\]
This inequality holds also without $\epsilon$, since it can been taken arbitrarily small.\\
Now let $\varphi$ be a test function for $V$ at $(t,m)$, i.e. such that $\varphi-V$
has a local minimum at $(t,m)$. Suppose for instance that, for some $\epsilon\in(0,1)$, $(t,m)$ is the minimum over all $(s,m')$ such that $\max(|s-t|,d_1(m',m))\leq\epsilon$. Then by relation (\ref{pcont}), we have, for all $s\in[0,T]$ such that $s-t\leq\epsilon^2/4d^2$, 
\[\begin{array}{rl}
\varphi(s,p^{t,m}_s)-\varphi(t,m)\geq & V(s,p^{t,m}_s)-V(t,m)\\
\geq & -\int_t^sH(r,p^{t,m}_r)dr.
\end{array}\]
Finally, using point $1)$ of Lemma \ref{techlemma}, we just have to divide by $s-t$ and let $s$ tend to $t$ to obtain
\[ \partial_t\varphi(t,m)+\frac{1}{2}\int_{\R^d}\di [D_m\varphi (t,m,x)]dm(x)  +H(t,m)\geq 0.\]
\end{proof}

\pa We already know that $V$ is continuous and convex in $m$. The next step is to prove that $V$ is smaller than any other bounded, continuous, convex subsolution of (\ref{eqH}). To this aim, we need two lemmas. The first is a comparison theorem between a solution and a subsolution of \eqref{eq3}.

\begin{Lemma}
\label{clas}
Let $F:[0,T]\times\PR_2\rightarrow\R$ be such that assumptions $(A1)$ hold true and let $\psi:\PR_2\rightarrow\R$  satisfy $(A2)$. 
Let $t_1\in[0,T]$, and define $\phi(t,m)=\psi(p^{t,m}_{t_1})+\int_t^{t_1}F(r,p^{t,m}_r)dr$. 
\pa
Let $U$ be a bounded, continuous subsolution of (\ref{eq3}) which satisfies, for all $m\in\PR_2$, $U(t_1,m)\leq \psi(m)$. Then $U\leq \phi$ on $[0,t_1]\times\PR_2$.
\end{Lemma}
 
\begin{proof} 
Let us suppose  that the assertion is wrong. Then we can find some $\gamma>0$ and $\epsilon>0$ such that
\begin{equation}
\label{supsub}
\sup_{(t,m)\in[0,t_1]\times\PR_2}\left( U(t,m)-\phi(t,m)-\gamma(t_1-t)-\epsilon|m|^2_2\right):=S>0
\end{equation}
Let $(t^n,m^n)$ be a maximizing sequence. We can extract a subsequence, without changing the notation, such that $(t^n)\rightarrow\bar t<t_1$.
\\
Further, from the assumption on $\psi$ and $F$, it follows that $\phi$ has almost linear growth in $m$ uniformly in $t$: there exists $C'>0$ such that, for all $(t,m)$,
\[ |\phi(t,m)|\leq C'(1+|m|_1).\]
Now, since $U$ is bounded and $|m|_1\leq |m|_2$, no term can compensate $\epsilon|m|_2^2$, i.e. the maximizing sequence has to be bounded :
there is some $K>0$ such that, for $n$ sufficiently large, $|m^n|_2^2\leq K$.\\
It follows from point 3 in Lemma \ref{dro} that there exists a subsequence that  converges to some $\bar m\in\PR_2$. \\
Since $U$ is upper semi-continuous by assumption, and $\phi$ is l.s.c. because $m\mapsto |m|^2_2$ is lower semi-continuous and $\psi$ is continuous, it follows that the supremum in (\ref{supsub}) is a maximum attained at $(\bar t,\bar m)$ with
\[ U(\bar t,\bar m)-\phi(\bar t,\bar m)-\gamma(t_1-\bar t)-\epsilon|\bar m|_2^2= S>0.\]
Set $\varphi(t,m):=\phi(t,m)+\gamma(t_1-t)+\epsilon|m|^2_2$.\\
We have
\[ U(\bar t,\bar m)-\varphi(\bar t,\bar m)=\max_{(t,m)\in[0,t_1]\times\PR_2}\left(U(t,m)-\varphi(t,m)\right).\]
Moreover $\varphi$ satisfies the regularity assumptions (A3) ( see Lemma \ref{techlemma} and Remark \ref{remarktestfunctions}).
Therefore the criterion for subsolutions applies:
\[ \frac{\partial\varphi}{\partial t}(\bar t,\bar m)+\frac{1}{2} \int_{\R^d}\di D_m\varphi(\bar t ,\bar m, x)d\bar m(x)+F(\bar t,\bar m)\geq 0.\]
But \[\frac{\partial\varphi}{\partial t}(\bar t,\bar m)=\frac{\partial\phi}{\partial t}(\bar t,\bar m)-\gamma\]
and
\[ \di D_m\varphi(\bar t,\bar m,x)=\di D_m\phi(\bar t,\bar m,x)-2d\epsilon.\]
This, together with (\ref{eqphi}) leads to the impossible relation
\[ -\gamma-d\epsilon\geq 0.\]
 \end{proof}

 \pa The following construction shall be crucial in the elaboration of a special test function for \eqref{eq3}.

\begin{Lemma}
\label{g}
Let $m_1\in\PR_2$ and $\delta>0$ be fixed. For all $m\in\PR_2$, we set
\[ \psi_\delta(m)= \int_{\R^d}\sqrt{\delta e^{-|x|^2}+|x|^2\left(\rho_\delta*m(x)-\rho_\delta*m_1(x)\right)^2}dx-(2\pi)^{\frac{d}{2}}\sqrt\delta.\]
Then 
\begin{enumerate}
\item $\psi_{\delta}\geq 0$ and $\psi_\delta(m_1)=0$.
\item $\psi_\delta$ satisfies $(A2)$.
\item for any $\nu>( (2\pi)^{\frac{d}{2}}+2\sqrt d)\sqrt{\delta}$, there exists $\alpha>0$ such that $d_1(m,m_1)\geq \nu\Rightarrow \psi_\delta\geq \alpha$.
\end{enumerate}
\end{Lemma}
\begin{proof} 
1. The first item is obvious.\\
2. Let us show that $\psi_\delta$ is continuous.\\
For $(m^n,n\geq 1)$ and $m\in\PR_2$ such that $d_1(m^n,m)\rightarrow 0$, it holds that (see the proof of 1.2 in Lemma \ref{techlemma})
\[ \lim_{K\rightarrow\infty}\sup_n\int_{|x|\geq K}|x|\rho_\delta*m^n(x)dx=0.\]
Given $\epsilon>0$, let us choose $K$ such that
\[ \sup_n\int_{|x|\geq K}|x|\rho_\delta*m^n(x)dx\leq\epsilon \mbox{ and }
\int_{|x|\geq K}|x|\rho_\delta*m(x)dx\leq\epsilon.\]
As $m^n$ converge weakly to $m$, we have simple convergence of $\rho_\delta*m^n(x)$ to $\rho_\delta*m(x)$. Remark now that
$y\mapsto \sqrt{\delta e^{-|x|^2}+|y|^2}$ is Lipschitz, uniformly in $x$. It follows that for some constant $C'$
\[ \begin{array}{rl}
|\psi_\delta(m^n)-\psi_\delta(m)|\leq &C'\int_{\R^d}|x||\rho_\delta*m^n(x)-\rho_\delta*m(x)|dx\\
&\leq 2C'\epsilon+C'\int_{|x|\leq K}|x||\rho_\delta*m^n(x)-\rho_\delta*m(x)|dx.
\end{array}\]
The last term goes to zero by bounded convergence and the conclusion follows as $\epsilon$ was chosen arbitrarily.
\pa We also have 
\[\psi_\delta(m) \leq \int_{\R^d}|x||\rho_\delta*m(x)-\rho_\delta*m_1(x)| dx \leq |m|_1+|m_1|_1+2\sqrt{d\delta},\]
which proves the linear growth property required by $(A2)$.
\pa
Using dominated convergence, we obtain that
\[ \frac{\delta\psi_\delta}{\delta m}(m,y)= \int_{\R^d} \frac{|x|^2 (\rho_{\delta}*m(x)-\rho_{\delta}*m_1(x)) \rho_{\delta}(x-y)}{\sqrt{\delta e^{-|x|^2}+|x|^2\left(\rho_\delta*m(x)-\rho_\delta*m_1(x)\right)^2}}dx.\]
Using that 
\[\left|\frac{|x| (\rho_{\delta}*m(x)-\rho_{\delta}*m_1(x)}{\sqrt{\delta e^{-|x|^2}+|x|^2\left(\rho_\delta*m(x)-\rho_\delta*m_1(x)\right)^2}}\right|\leq 1,\]
it is easy to see that $\frac{\delta\psi_\delta}{\delta m}$ satisfy the regularity assumptions required in $(A2)$.\\
3. By Lemma \ref{dp} and Lemma \ref{dro} it holds that
\[\begin{array}{rl}
d_1(m,m_1)\leq & d_1(p^{0,m}_\delta,p^{0,m_1}_\delta)+2\sqrt{d\delta}\\
\leq & \int_{\R^d}|x||\rho_\delta*m(x)-\rho_\delta*m_1(x)|dx+2\sqrt{d\delta}\\
\leq & \psi_\delta(m)+((2\pi)^{\frac{d}{2}}+2\sqrt d)\sqrt{\delta}.
\end{array}\]
The result follows.
\end{proof}

\begin{Proposition}
\label{increq}
Suppose that $F$ satisfies the assumptions (A1). Let $U$ be a bounded, continuous subsolution of (\ref{eq1}). For all $(t_0,m_0)\in[0,T]\times\PR_2$, for all $s\in[t_0,T]$, we have
 \[ U(s,p^{t_0,m_0}_s)-U(t_0,m_0)+\int_t^sF(r,p^{t_0,m_0}_r)dr\geq 0.\]
\end{Proposition}
 
\begin{proof}
Let  $t_1\in[t_0,T]$, and $\epsilon>0$. We set $m_1:=p^{t_0,m_0}_{t_1}$.
Since $U$ is continuous, there exists $\eta>0$ such that, for all $m\in\PR_2$, 
\[ d_1(m,m_1)\leq\eta\Rightarrow |U(t_1,m)-U(t_1,m_1)|\leq\epsilon.\]
Let us choose $\delta>0$ such that $((2\pi)^{\frac{d}{2}}+2\sqrt d)\sqrt\delta<\eta$ and let $\psi_\delta$ be as in Lemma \ref{g}. Following Lemma \ref{g}, we can find $\alpha>0$ such that, for all $m\in\PR_2$, $d_1(m,m_1)\geq\eta\Rightarrow\psi_\delta(m)\geq\alpha$.\\
Set $K:=\sup_{m\in\PR_2}\left( U(t_1,m)-U(t_1,m_1)\right)+1$ and $\tilde\psi(m):=\frac{K}{\alpha}\psi_\delta(m), m\in\PR_2$, in order to get
\[ d_1(m,m_1)\geq\eta\Rightarrow \tilde\psi(m)\geq K.\]
This implies
\begin{equation}
\label{boule}
 \sup_{m\in\PR_2}\left( U(t_1,m)-\tilde\psi(m)\right)
=
\sup_{m\in\PR_2,d_1(m,m_1)\leq\eta}\left( U(t_1,m)-\tilde\psi(m)\right).
\end{equation}
Let $(m^n)_{n\geq 1}$ be a maximizing sequence for the right hand side of (\ref{boule}).
Since $\tilde\psi\geq 0$ and $U$ is bounded, the two sequences $(U(t_1,m^n))_n$ and $(\tilde\psi(m^n))_n$ are also bounded and we can find a subsequence (still denoted by $(m^n)$) such that $\tilde a:=\lim_{n}U(t_1,m^n)$ and $\tilde b:=\lim_n\tilde\psi(m^n)$ exist. Moreover these limits satisfy
\[ \tilde a-\tilde b=\sup_{m\in\PR_2}\left( U(t_1,m)-\tilde\psi(m)\right).\]
Finally, from the fact that, for all $n\geq 1$, $d_1(m^n,m_1)\leq\eta$, we deduce that
$|U(t_1,m^n)-U(t_1,m_1)|\leq\epsilon$
as well as
\begin{equation}
\label{epsilon}
 |\tilde a -U(t_1,m_1)|\leq\epsilon.
 \end{equation}
Now let $\psi(m):=\tilde a+\tilde\psi(m), m\in\PR_2$. Then
\[\begin{array}{rl}
\psi(m)-U(t_1,m)=&\tilde a+\tilde\psi(m)-U(t_1,m)\\
\geq & \tilde a+\tilde b-\tilde a=\tilde b\geq 0.
\end{array}
\]
It follows that Lemma \ref{clas} applies : for $\phi(t,m)=\psi(p^{t,m}_{t_1})+\int_t^{t_1}F(r,p^{t,m}_r)dr$, we get in particular
\[ U(t_0,m_0)\leq \phi(t_0,m_0)\]
or equivalently
\begin{equation}
\label{uleqphi}
 U(t_0,m_0)-\int_{t_0}^{t_1}F(s,p^{t_0,m_0}_s)ds\leq \psi(m_1).
 \end{equation}
Now remember that $\psi(m_1)=\tilde a+\tilde\psi(m_1)=\tilde a$.
Thus, combining (\ref{uleqphi}) with (\ref{epsilon}), we obtain
\[ U(t_0,m_0)-\int_{t_0}^{t_1} F(s,p^{t_0,m_0}_s)ds\leq U(t_1,m_1)+\epsilon.\]
This last $\epsilon$ being arbitrary, the result follows.
 \end{proof}

\pa
Now we come back to our original equation (\ref{eqH}) involving the Hamiltonian $H$. The purpose of what follows is to show that the result of Proposition \ref{increq} holds also for $F=H$
 even if $H$ doesn't satisfy the regularity assumptions (A1).\\
 
\pa The first Lemma is a result of type Stone-Weierstrass, which is strongly inspired by the lecture notes of P. Cardaliaguet \cite{CL} on the lecture of P.L. Lions \cite{CL}.

\begin{Lemma}
\label{stoneweierstrass}
Let $Q$ be a compact set in $\R^d$ and $\PR(Q)$ the set of probability measures on $Q$. 
A monomial on $[0,T]\times\PR(Q)$ is a map $P:[0,T]\times\PR(Q)\rightarrow\R$ of the form
\[ P(t,m)=t^k\Pi_{i=1}^n\int_Q\phi_i(x)dm(x),\]
with  $k,n\in\N$ and $\phi_1,\ldots,\phi_n\in C^\infty(Q)$.
We call polynomial any linear combination of monomials. The set $\Pi$ of polynomials is dense in $C^0([0,T]\times\PR(Q))$ endowed with the sup-norm.
\end{Lemma} 
 
 \begin{proof}
It is easy to see that $\Pi$ is a sub-algebra in $C^0([0,T]\times\PR(Q))$ and that $\Pi$ contains a non-zero constant: $P(t,m)\equiv 1$. $\Pi$ also separates points: indeed let $(t_1,m_1)\neq (t_2,m_2)$.  It is obvious that, if $t_1\neq t_2$, $P(t,m)=t$ gives $P(t_1,m_1)\neq P(t_2,m_2)$.
 Now, if $t_1=t_2$ but $m_1\neq m_2$, we can find a smooth map $\varphi$ such that 
$\int_Q\varphi dm_1\neq\int_Q\varphi dm_2$. This means that the monomial $P(t,m)=\int_Q\varphi dm$ separates the points $(t_1,m_1)$ and $(t_2,m_2)$.
Hence the Stone-Weierstrass-theorem applies and the result follows.
\end{proof}

\begin{Lemma}
\label{tildeH}
For any $K>0$ and $\epsilon>0$, there exists $\tilde H:[0,T]\times\PR_2\rightarrow\R$ satisfying (A1) such that, for all $t\in[0,T]$ and all $m\in\PR_2$ with $|m|_2^2\leq K$,
\[ |H(t,m)-\tilde H(t,m)|\leq \epsilon.\]
\end{Lemma}

\begin{proof} Take $R>\frac{4CK}\epsilon$ and define $K_r=\{ m \in \PR_2 \,|\,\mbox{Supp}(m)\subset\BR(2R)\}$ with $\mbox{Supp}(m)$ the support of $m$ and  $\BR(2R):=\{ x\in\R^d,|x|\leq 2R\}$. $K_r$ is weakly compact, and thus compact for $d_1$. It follows therefore from Lemma \ref{stoneweierstrass} that there exists a polynomial map $H_{R,\epsilon}: K_r \rightarrow\R $, such that, for all $t\in[0,T]$ and all $m\in K_r$,
\begin{equation}
\label{HR}
 | H_{R,\epsilon}(t,m)-H(t,m)|\leq\frac{\epsilon}{2}.
 \end{equation}
Now let $\varphi_R\in C^\infty(\R^d,\R^d)$, with
\[\varphi_R(x)=\left\{
\begin{array}{ll}
x,&\mbox{ if }x\in\BR(R),\\
0& \mbox{ if } x\in\BR(2R)^c
\end{array}
\right.\]
and $|\varphi_R(x)|\leq|x|$, for all $x\in\R^d$. Remark that, for all $m\in\PR_2$, $\varphi_R\sharp m$ has its support in $\BR(2R)$, where $\varphi_R\sharp m$ denotes the pushforward of $m$ by $\varphi_R$ defined by $\varphi_R\sharp m(A)=m(\varphi_R^{-1}(A))$ for all Borel subsets $A$ of $\R^d$.\\
Set $\tilde H_{R,\epsilon}(t,m)=H_{R,\epsilon}(t,\varphi_R\sharp m)$.
It is easy to see  that 
\[\frac{\delta \tilde H_{R,\epsilon}}{\delta m} (t,m,x)= \frac{\delta H_{R,\epsilon}}{\delta m}(t,\varphi_R\sharp m, \varphi_R(x)),\]
and therefore that $\tilde H_{R,\epsilon}$ satisfies (A1).
Since $H$ is Lipschitz in $m$, we have, for all $(t,m)\in[0,T]\times\PR_2$ such that $|m|_2^2 \leq K$,
\[\begin{array}{rl}
|H(t,m)-\tilde H_{R,\epsilon}(t,m)|\leq &|H(t,m)-H(t,\varphi_R\sharp m)|+|H(t,\varphi_R\sharp m)-\tilde H_{R,\epsilon}(t,m)| \\
\leq & Cd_1(m,\varphi_R\sharp m)+\frac{\epsilon}2\\
\leq & C\int_{\R^d}|x-\varphi_R(x)|dm(x)+\frac{\epsilon}2\\
\leq & 2C\int_{\BR(R)^c}|x|dm(x)+\frac{\epsilon}2\\
\leq & \frac{2C}{R}\int_{\R^d}|x|^2dm(x)+\frac{\epsilon}2\leq \epsilon.
\end{array}
\]
The result follows.
\end{proof}

\pa We are ready now to claim that Proposition \ref{increq} holds true for $F=H$:
\begin{Proposition}
\label{incr}
Let $U$ be a bounded continuous subsolution of (\ref{eqH}). Then, for all $(t_0,m_0)\in[0,T]\times\PR_2$,
\[ U(s,p^{t_0,m_0}_s)-U(t_0,m_0)+\int_t^sH(r,p^{t_0,m_0}_r)dr\geq 0.\]
\end{Proposition} 
 
 \begin{proof} Fix $(t_0,m_0)\in[0,T]\times\PR_2$. It holds that, for any $r\in[t_0,T]$, 
 \[| p^{t_0,m_0}_r|_2^2\leq K:=|m_0|^2_2+dT.\]
 Using Lemma \ref{tildeH}, we can find, for any arbitrarily small $\epsilon>0$, a map $\tilde H$ which satisfies (A1) and such that  $|\tilde H(r,p^{t_0,m_0}_r)-H(r,p^{t_0,m_0}_r)|\leq \epsilon$ for all $r\in[t_0,T]$. \\
 If $U$ is a subsolution of (\ref{eqH}), then $\tilde U(s,m):=U(s,m)+\epsilon s$ is a subsolution of (\ref{eq1}) for $F=\tilde H$. From Lemma \ref{increq}, it follows that
 \[\tilde U(s,p^{t_0,m_0}_s)-\tilde U(t_0,m_0)+\int_{t_0}^s\tilde H(r,p^{t_0,m_0}_r)dr\geq 0,\]
 and therefore
 \[U(s,p^{t_0,m_0}_s)-U(t_0,m_0)+\int_{t_0}^sH(r,p^{t_0,m_0}_r)dr\geq -2\epsilon(s-t_0).\]
Since this relation holds true for all $\epsilon>0$, the result follows.
 \end{proof}
 
\pa The above result permits us finally to conclude by the main theorem, which characterizes the value function $V$ as the largest bounded, continuous subsolution, which  is convex in $m$ and vanishes at time $T$.
  \begin{Theorem}
Let $U$ be a bounded, continuous subsolution of (\ref{eqH}) which is convex in $m$ and satisfies $U(T,m)=0$ for all $m\in\PR_2$. Then $U\leq V$. 
 \end{Theorem}
 
 \begin{proof}
 Let $(t_0,m_0)\in[0,T]\times\PR_2$ and $M\in\MR(t_0,m_0)$.
 Consider the regular time grid $t_0<t_1<\ldots<t_N=T$, with $t_k=t_0+k\frac{(T-t_0)}N$.\\
 By Proposition \ref{incr}, we have
 \begin{equation}
 \label{utp}
 U(t_{k+1},p^{t_k,M_{t_k}}_{t_{k+1}})-U(t_k,M_{t_k})+\int_{t_k}^{t_{k+1}}H(r,p^{t_k,M_{t_k}}_r)dr\geq 0.
 \end{equation}
From the definition of $\MR(t_0,m_0)$ and since $U$ is convex in $m$, it follows that
\[ U(t_{k+1},p^{t_k,M_{t_k}}_{t_{k+1}})=U\left(t_{k+1},\E[M_{t_{k+1}}|\FR^M_{t_k}]\right)
\leq \E[ U(t_{k+1},M_{t_{k+1}})|\FR^M_{t_k}].\]
And therefore
\begin{equation}
\label{utk}
\E\left[ U(t_k,M_{t_k})-U (t_{k+1},M_{t_{k+1}})|\FR^M_{t_k}\right]\leq\E\left[\int_{t_k}^{t_{k+1}}H(r,p^{t_k,M_{t_k}}_r)dr|\FR^M_{t_k}\right].
\end{equation} 
Since, by Lemma \ref{HLipsch}, $H$ is Lipschitz in $m$, with Lipschitz constant $C$, we have, for any $r\in [t_k,t_{k+1}]$,
\begin{equation}
\label{h}
 \E\left[ |H(r,p^{t_k,M_{t_k}}_r)-H(r,M_{t_k})|\big|\FR^M_{t_k}\right]\leq C\E\left[d_1(p^{t_k,M_{t_k}}_r,M_{t_k})\big|\FR^M_{t_k}\right].
 \end{equation}
By Lemma \ref{dp}, we have 
\[d_1(p^{t_k,M_{t_k}}_r,M_{t_k}) \leq \sqrt{d(r-t_k)}.\]
With (\ref{h}), this gives us for $H$:
\[\begin{array}{rl}
E\left[\int_{t_k}^{t_{k+1}}H(r,p^{t_k,M_{t_{k}}}_r)dr\big|\FR^M_{t_k}\right] \leq &E\left[\int_{t_k}^{t_{k+1}}H(r,M_{t_k})dr\big|\FR^M_{t_k}\right]+\frac{2C\sqrt d}3(t_{k+1}-t_k)^{3/2}
\end{array}\]

\noi Coming back to (\ref{utk}), thanks to the choice of the path of the time grid, we can write
\[
\E\left[U(t_k,M_{t_k}) -U(t_{k+1},M_{t_k+1})\right]\leq \E\left[\int_{t_k}^{t_{k+1}} H(r,M_{t_k})dr\right]+\frac{2C\sqrt d}3(t_{k+1}-t_k)^{3/2}.
\]
Summing up over all $k$, we get
\[ \E\left[ U(t_0,m_0)\right]=\E\left[ U(t_0,m_0)- U(T,M_T)\right]\leq \E\left[\int_{t_0}^TH(r,\tilde{M}_r)dr\right]+ \frac{2C\sqrt d(T-t_0)^{3/2}}{3 \sqrt{N}},
\]
with $\tilde M_r$ equal to $M_{t_k}$ if  $r\in[t_k,t_{k+1})$, for all $k$.
As the mesh of the grid goes to zero, the process $s \rightarrow H(s,\tilde{M}_s)$ converges almost surely for the Skorokhod topology  to $s \rightarrow H(s,M_s)$ (see e.g. Proposition VI.6.37 in \cite{JacodShiryaev}), so that:
\[ U(t_0,m_0)\leq \E\left[\int_{t_0}^TH(r,M_r)dr\right].\]
And, finally, since $M\in\MR(t_0,m_0)$ has been chosen arbitrarily, the result follows.
 \end{proof}
 
 \pa {\bf Acknowledgements:} We thank Pierre Cardaliaguet for very fruitful discussions.\\
 We thank the referees for their careful reading and their pertinent remarks.\\
 
 \pa This work has been partially supported by the French National Research Agency ANR-16-CE40-0015-01 MFG.

\end{document}